\RequirePackage{ifthen}
\newboolean{MPA}
\setboolean{MPA}{false}

\ifthenelse {\boolean{MPA}}
{
\documentclass{svjour3}
\smartqed
\usepackage[margin=1.3in]{geometry}
} {
\documentclass[11pt,letterpaper]{article}
\usepackage[margin=1in]{geometry}
}

\usepackage{graphicx} 

\usepackage{amsmath}
\usepackage{amssymb}
\ifthenelse {\boolean{MPA}} {} {\usepackage{amsthm}}
\usepackage{hyperref}
\usepackage{xcolor}
\usepackage{epsfig}
\usepackage{tikz}
\usetikzlibrary{calc}
\usepackage{enumitem}

\usepackage{enumitem}
\setlist{leftmargin=6mm,nolistsep,noitemsep}

\DeclareMathOperator{\env}{env}

\def\tw{{\rm tw}}
\def\poly{{\rm poly}}

\newcommand{\ie}{i.e., }

\def\R{{\mathbb R}}

\def\Q{{\mathbb Q}}
\def\X{{\mathcal X}}
\def\A{{\mathcal A}}

\def\H{{\mathcal H}}

\ifthenelse {\boolean{MPA}}
{

\let\svjproof\proof
\let\endsvjproof\endproof
\renewenvironment{proof}[1][Proof]
{\def\proofname{#1}\svjproof}
{\qed \endsvjproof}

\newcounter{claim} 
\renewenvironment{claim}[1][]
{\refstepcounter{claim} \begin{trivlist} \item[] {\bf Claim~\theclaim}\space#1 \itshape}
{\end{trivlist}}

\newenvironment{cpf}
{\begin{trivlist} \item[] {\em Proof of claim }}
{$\hfill\diamond$ \end{trivlist}}

\journalname{Mathematical Programming A}

\newtheorem{observation}{Observation}

} {

\newtheorem{theorem}{Theorem}
\newtheorem{corollary}{Corollary}
\newtheorem{proposition}{Proposition}

\newtheorem{lemma}{Lemma}
\newtheorem{claim}{Claim}

\theoremstyle{remark}
\newtheorem{remark}{Remark}

\theoremstyle{definition} 
\newtheorem{example}{Example}

\newenvironment{cpf}
{\begin{trivlist} \item[] {\em Proof of claim. }}
{$\hfill\diamond$ \end{trivlist}}

}

\newcommand{\NP}{\mathcal {NP}}
\renewcommand{\L}{\mathcal L}

\newcommand{\ch}{\mathrm{ch}}

\date{}

\title{Treewidth and the complexity of box-constrained quadratic programs}

\ifthenelse {\boolean{MPA}}
{
\titlerunning{Treewidth and the complexity of box-constrained quadratic programs}

\author{Alberto Del Pia \and Aida Khajavirad}

\institute{Alberto Del Pia \at
              Department of Industrial and Systems Engineering \& Wisconsin Institute for Discovery, 
              University of Wisconsin-Madison.
              E-mail: {\tt delpia@wisc.edu}.
           \and
           Aida Khajavirad \at
              Department of Industrial and Systems Engineering,
              Lehigh University.
              E-mail: {\tt aida@lehigh.edu}.
}
}
{
\author{Alberto Del Pia
\thanks{Department of Industrial and Systems Engineering \& Wisconsin Institute for Discovery,
             University of Wisconsin-Madison.
             E-mail: {\tt delpia@wisc.edu}.
             }
\and
Aida Khajavirad
\thanks{Department of Industrial and Systems Engineering,
             Lehigh University.
             E-mail: {\tt aida@lehigh.edu}.
             }
}
}

\begin{document}

\maketitle
\begin{abstract}
    We consider the problem of minimizing a sparse quadratic function over the unit hypercube. In binary quadratic programming, treewidth of the interaction graph is a central parameter for tractability: bounded treewidth yields polynomial-time solvability. Motivated by this fact, we investigate whether treewidth plays a similar role when the binary domain is replaced by the unit hypercube. We show that the situation is strikingly different. If the interaction graph is a forest, we give a strongly polynomial-time algorithm based on dynamic programming and a structural analysis of the resulting univariate value functions, which are shown to be concave and piecewise quadratic with linearly many pieces. However, the problem becomes strongly $\NP$-hard already when the interaction graph has treewidth two. We then identify substantially more general polynomial-time solvable classes whose interaction graphs may have unbounded treewidth. Our approach exploits the fact that there exists an optimal solution in which every variable with a nonpositive coefficient for its square term is binary-valued, thereby separating the problem into a combinatorial part and a genuinely continuous part. We show that tractability can be recovered by controlling the complexity of these two parts and their interaction, rather than the treewidth of the entire interaction graph.
\ifthenelse {\boolean{MPA}}
{
} {}
\end{abstract}

\medskip
\ifthenelse {\boolean{MPA}}
{}{
}

\section{Introduction}

We consider a box-constrained quadratic program:
\begin{align}\label{eq:QP}
\tag{QP}
\min \quad & x^\top Q x + c^\top x \\
{\rm s.t.} \quad & x \in [0,1]^n, \nonumber
\end{align}
where $c \in \Q^n$ and $Q=(q_{ij}) \in \Q^{n\times n}$ is a symmetric matrix. 
It is well known that Problem~\eqref{eq:QP} is $\NP$-hard in general, as it contains the maximum cut problem as a special case~\cite{HorTuy96}. In this paper, we are interested in obtaining sufficient conditions under which Problem~\eqref{eq:QP} can be solved in polynomial time. Throughout this paper, we consider the standard Turing bit model as our model of computation. If the number of variables is fixed, then the problem can be solved in strongly polynomial time using a simple face enumeration argument \cite{Vav90,dPDeyMol17MPA}.
If $Q$ is positive semidefinite, then Problem~\eqref{eq:QP} is a convex optimization problem, and can be solved in polynomial time using the ellipsoid method~\cite{TarKha79}. If $Q$ is positive semidefinite and tridiagonal, Pang and Han~\cite{PanHan23} give a strongly polynomial algorithm for Problem~\eqref{eq:QP}. In~\cite{hladik21}, the authors prove that Problem~\eqref{eq:QP} can be solved in polynomial time if the rank of $Q$ is fixed.
If $q_{ii} \leq 0$ for some $i \in [n]:=\{1,\ldots,n\}$, then it can be checked that there exists an optimal solution $x^*$ of Problem~\eqref{eq:QP} at which $x^*_i \in \{0,1\}$. It then follows that if $q_{ii} \leq 0$ for all $i \in [n]$, to find an optimal solution of Problem~\eqref{eq:QP},  we can solve the \emph{binary quadratic program}:
\begin{align}\label{eq:BQP}
\tag{BQP}
\min \quad & x^\top Q x + c^\top x \\
{\rm s.t.} \quad & x \in \{0,1\}^n. \nonumber
\end{align}
Problem~\eqref{eq:BQP} has been extensively studied by the discrete optimization community~\cite{Pad89,CraHanJau90,BorHam02}. 
If the objective function is submodular, \ie $q_{ij} \leq 0$ for all $1\leq i<j \leq n$, then Problem~\eqref{eq:BQP} can be solved in strongly polynomial time~\cite{schrijver2000}.
An important line of research exploits the sparsity of the objective function of Problem~\eqref{eq:BQP} to obtain sufficient conditions for its polynomial-time solvability.  
To this end, we define, for Problem~\eqref{eq:QP} or Problem~\eqref{eq:BQP}, the \emph{interaction graph} as the graph $G=(V,E)$ with vertex set $V:=[n]$, containing one vertex $i$ for each variable $x_i$, and with edge set $E$, where two distinct vertices $i$ and $j$ are adjacent if and only if $q_{ij}\neq0$. The most general result in this literature states that if the graph $G$ has bounded treewidth, then Problem~\eqref{eq:BQP} can be solved in strongly polynomial time~\cite{CraHanJau90}. In fact, the results of~\cite{ChaSreHar08,dpAk24} imply that bounded treewidth is essentially a necessary and sufficient condition for polynomial-time solvability of Problem~\eqref{eq:BQP}. 
See~\cite{Cra88tr,Mic17,Pun22,dP26MPA} for additional classes of polynomial-time solvable binary quadratic programs.

These results naturally raise the question of whether treewidth plays a similar role when the binary domain is replaced by the unit hypercube. Our goal is therefore to understand \emph{the relationship between the treewidth of the interaction graph and the computational complexity of sparse box-constrained quadratic programs}.
The work most closely related to ours is Khajavirad~\cite{Aida26p}, who studies polynomial-time solvability of sparse polynomial optimization problems over the unit hypercube. Specializing to quadratic objectives, the paper observes that variables with nonpositive diagonal quadratic coefficients can be restricted to binary values at optimality, and exploits this structure to separate the problem into a discrete part and a collection of continuous components. The continuous components are eliminated individually under a condition controlling their degree of nonconvexity, yielding a structured binary quadratic optimization problem. Polynomial-time solvability is then obtained under suitable bounds on the treewidth of the resulting interaction graph. In this paper, we substantially extend this result. See~\cite{KimKoj03,SojLav14,QiuYil24,DeyIda25,Aida26} for sufficient conditions under which certain semidefinite programming relaxations of Problem~\eqref{eq:QP} are tight; namely, the optimal value of the relaxation coincides with that of the original problem.

\medskip
Our main contributions can be summarized as follows.
\medskip

\begin{itemize}[leftmargin=0.8cm]

\item [$(i)$] We prove that Problem~\eqref{eq:QP} can be solved in strongly polynomial time if its interaction graph has treewidth at most one, that is, if it is a forest. Our algorithm is based on dynamic programming over the forest and a structural analysis of the resulting univariate value functions, which are shown to be concave and piecewise quadratic with linearly many pieces. The resulting algorithm performs $O(n^2)$ arithmetic operations and comparisons (see Theorem~\ref{thm:main}).

\item [$(ii)$] We prove that the above tractability result is essentially sharp with respect to treewidth: Problem~\eqref{eq:QP} is strongly $\NP$-hard when the interaction graph has treewidth two. The hardness result holds even when all coefficients are integral and uniformly bounded, namely $\|Q\|_{\max}\leq 5$ and $\|c\|_{\infty}\leq 4$ (see Theorem~\ref{thm:strong-NP-hardness-tw2}).

\item [$(iii)$] We obtain more general sufficient conditions for polynomial-time solvability of Problem~\eqref{eq:QP} that apply to interaction graphs of potentially unbounded treewidth. We first exploit the fact that there exists an optimal solution in which every variable with a nonpositive diagonal quadratic coefficient is binary-valued, thereby separating the problem into a combinatorial part and a continuous part. We then derive tractability conditions based on the complexity of the connected components induced by these two parts and on the structure of their interaction (see Theorem~\ref{th:oneBlock} and Theorem~\ref{th:oneSwitch}).
\end{itemize}
\medskip

Our results reveal a sharp transition in the role of treewidth for box-constrained quadratic programs: if the interaction graph is a forest, the problem is strongly polynomial-time solvable, while strong $\NP$-hardness already arises at treewidth two. This contrasts with binary quadratic programming, which is polynomial-time solvable on graphs of bounded treewidth. This difference can be attributed to the genuinely continuous part of the problem: variables with nonpositive square coefficients can be chosen binary at optimality, whereas variables with positive square coefficients cannot, in general, be reduced in this way. This simple observation separates the problem into a combinatorial part and a genuinely continuous part. Our results then show that tractability can be recovered by controlling the complexity of these two parts and their interaction, rather than the treewidth of the entire graph. Perhaps surprisingly, our results indicate that, at treewidth two, enlarging the feasible set from $\{0,1\}^n$ to $[0,1]^n$ while keeping the same quadratic objective function can turn a polynomial-time solvable problem into a strongly $\NP$-hard one.

It is important to note that all of the results in this paper extend directly to the mixed-binary setting; that is, the problem of minimizing a quadratic function over $x_i \in [0,1]$ for all $i \in I_1$ and $x_i \in \{0,1\}$ for all $i \in I_2$, where $I_1$ and $I_2$ partition $[n]$. Indeed, consider a binary variable $x_i$ for some $i \in I_2$. 
Replace $q_{ii} x^2_i+ c_i x_i$ in the objective function by $(q_{ii}+c_i+1)x_i-x_i^2$, and relax $x_i\in\{0,1\}$ to $x_i\in[0,1]$. 
Apply this transformation successively to all binary variables.
It can be checked that the set of optimal solutions of the original mixed-binary problem coincides with that of the box-constrained quadratic program. The key to this transformation is that since we only modify diagonal and linear coefficients, the interaction graph remains unchanged.

\paragraph{Organization.} The remainder of the paper is structured as follows. In Section~\ref{sec:tw1} we present a strongly polynomial-time algorithm for box-constrained quadratic programs whose interaction graph is a forest. In Section~\ref{sec: NP-hard}, we show that if the treewidth of the interaction graph is two, then Problem~\eqref{eq:QP} is strongly $\NP$-hard, in general. In Section~\ref{sec: poly} we obtain sufficient conditions for polynomial-time solvability of box-constrained quadratic programs whose interaction graphs may have unbounded treewidth.

\paragraph{Notation.} Throughout the paper, we make use of the following notation. For two subsets $S,S'\subseteq[n]$, we denote by $Q_{S,S'}$ the submatrix of $Q$ with rows indexed by $S$ and columns indexed by $S'$, and we abbreviate $Q_S:=Q_{S,S}$ for the principal submatrix of $Q$ indexed by $S$. In particular, for $i\in[n]$, the vector $Q_{S,i}$ is the column of $Q$ indexed by $i$, restricted to the rows in $S$. For $S\subseteq[n]$, we denote by $c_S$ and $x_S$ the restriction of $c$ and $x$, respectively, to the coordinates in $S$. Given a graph $G=(V,E)$ and $S\subseteq V$, we denote by $G_S$ the subgraph of $G$ induced by $S$.
We denote by $\langle\cdot\rangle$ the \emph{lengths} (also known as encoding lengths or sizes) of rational numbers, vectors and matrices, as defined in~\cite{GroLovSch88,SchBookIP}; the \emph{input length} of Problem~\eqref{eq:QP} is $\L:=\langle Q\rangle+\langle c\rangle$.
Throughout the paper, $\log$ denotes the base-$2$ logarithm.

\section{Treewidth one: A strongly polynomial-time algorithm}
\label{sec:tw1}
In this section, we focus on the case where the interaction graph $G$ of Problem~\eqref{eq:QP} has treewidth at most one, that is, $G$ is a forest, and we present a strongly polynomial-time algorithm for it.
Recall that an algorithm is \emph{strongly polynomial} when the number of arithmetic operations and comparisons that it performs is bounded by a polynomial in the number of input numbers, independently of their values, and every intermediate number has length polynomial in the input length.

\begin{theorem}
\label{thm:main}
Assume that the interaction graph $G$ is a forest. Then an optimal solution of Problem~\eqref{eq:QP} can be found by a strongly polynomial algorithm that performs $O(n^2)$ arithmetic operations and comparisons.
\end{theorem}


Theorem~\ref{thm:main} is proved by a leaf-to-root dynamic programming algorithm whose univariate value functions are concave and piecewise quadratic, with a number of arcs that adds, rather than multiplies, at vertices with several children. The algorithm performs only arithmetic operations $(+,-,\times,\div)$ and comparisons on rational numbers; throughout this section, an \emph{operation} is one of these. The value functions are described by rational numbers and by algebraic numbers of degree at most two over $\Q$, which are in general irrational and are represented exactly as $\tau_0+\tau_1\sqrt D$ with rational $\tau_0,\tau_1,D$; auxiliary numbers of degree at most four, represented by nested square roots, arise while a single value function is computed and are discarded afterwards. No root is ever extracted: comparing any two of these numbers takes $O(1)$ operations.

We denote by $f(x)$ the objective function of Problem~\eqref{eq:QP}. Problem~\eqref{eq:QP} decomposes on the connected components $V_1,\dots,V_k$ of $G$, and since $G$ is a forest, each resulting subproblem is an instance of Problem~\eqref{eq:QP} whose interaction graph is a tree. As these components can be identified with $O(n^2)$ operations and $\sum_{i=1}^k|V_i|^2\le n^2$, it suffices to prove Theorem~\ref{thm:main} when $G$ is a tree. Henceforth, we assume that $G$ is a tree, and we write $T=(V,E):=G$.


The rest of this section is organized as follows. Section~\ref{sec:valuefunctions} sets up the dynamic program; Section~\ref{sec:envelope} is a self-contained study of convex envelopes of piecewise quadratic functions with concave kinks; Section~\ref{sec:arithmetic} proves that the value functions are piecewise quadratic with a linear number of arcs, whose coefficients are rationals of polynomial length; Section~\ref{sec:computation} shows how to compute them exactly and proves Theorem~\ref{thm:main}. Section~\ref{sec:relation} relates our approach to the literature, and in particular compares it with that of~\cite{BhaFatGomKuc26}. Finally, Section~\ref{sec:quartic} shows that a similar tractability result does not hold for higher degree polynomials; namely, minimizing a quartic over the unit hypercube is strongly $\NP$-hard even when the interaction graph is a path.

\subsection{Value functions and the dynamic program}
\label{sec:valuefunctions}

Root $T$ at an arbitrary vertex $r$; for $v\ne r$, write $p(v)$ for the \emph{parent} of $v$, write $\ch(v):=\{u\in V: p(u)=v\}$ for the set of \emph{children} of $v$, and write $T_v$ for the vertex set of the subtree rooted at $v$. The \emph{depth} of $u$ is the number of edges on the path from $u$ to $r$, and the \emph{height} of $v$ is the largest number of edges on a path from $v$ to a vertex of $T_v$. Separate $f$ along the tree:
\[
f(x) = \sum_{v\in V}\big(q_{vv}x_v^2 + c_vx_v\big) + \sum_{v\ne r} 2q_{p(v)v}\,x_{p(v)}x_v .
\]
We abbreviate $s_v:=2q_{p(v)v}$ for $v\ne r$, which is nonzero since $\{p(v),v\}$ is an edge of $G$, and set $s_r:=0$. For $v\in V$ let $e_v\in\mathbb R^{T_v}$ be the $v$-th unit vector, and define
\begin{equation}
\label{eq:Fv}
F_v(y;t):=y^\top Q_{T_v}\,y+\big(c_{T_v}+s_v t\,e_v\big)^\top y,
\qquad y\in\mathbb R^{T_v},\ t\in\mathbb R.
\end{equation}
Then $F_v(\cdot\,;t)$ is the objective of the subproblem on the subtree rooted at $v$, with the parent variable set to $t$, and $F_r(y;0)=f(y)$. Since the sets $T_j$, $j\in\ch(v)$, partition $T_v\setminus\{v\}$, and the edges of $T$ between $v$ and $T_v\setminus\{v\}$ are the edges $\{v,j\}$, $j\in\ch(v)$, we have
\begin{equation}
\label{eq:Fsplit}
F_v(y;t)=\big(q_{vv}y_v^2+c_vy_v\big)+\sum_{j\in \ch(v)}F_j\big(y_{T_j};y_v\big)+s_v\,t\,y_v .
\end{equation}

We define two families of univariate functions, by induction on the height of $v$: the \emph{value function} $\mu_v$, one for each $v\ne r$, and an auxiliary function $\varphi_v$, one for each $v\in V$, given by
\begin{equation}
\label{eq:message}
\varphi_v(x) \;:=\; q_{vv}x^2 + c_vx + \sum_{j\in \ch(v)} \mu_j(x),
\qquad
\mu_v(t) \;:=\; \min_{x\in[0,1]}\ \Big[\varphi_v(x) + s_v\,t\,x\Big].
\end{equation}
The functions $\varphi_v$ are defined on all of $\mathbb R$, but only their restriction to $[0,1]$ enters \eqref{eq:message}.

We will need some terminology for univariate functions. We call a continuous function $g:I\to\mathbb R$ on an interval $I\subseteq\mathbb R$ (possibly unbounded) \emph{piecewise quadratic} (pw-quadratic) if $I$ splits into finitely many nondegenerate intervals (\emph{arcs}), each closed in $I$ and possibly unbounded, on each of which $g$ is a polynomial of degree $\le 2$; the shared arc endpoints interior to $I$ are \emph{breakpoints}. We always take the arcs to be maximal, so that no two consecutive ones carry the same polynomial; the breakpoints are then exactly the points of $I$ at which $g$ fails to be a polynomial of degree $\le2$ on a neighborhood, and the decomposition into arcs is unique. A breakpoint $x$ is a \emph{kink} if $g'(x^-)\ne g'(x^+)$; a kink is \emph{concave} if $g'(x^-)>g'(x^+)$. An arc is \emph{strictly convex}, \emph{affine}, or \emph{strictly concave} according to the sign of its quadratic coefficient.

We first record a standard fact about concave functions; see Theorems~10.1 and~24.1 in~\cite{Roc70}, applied to $-g$.

\begin{lemma}
\label{lem:concavejumps}
A concave function $g:\mathbb R\to\mathbb R$ is continuous, has one-sided derivatives everywhere, and satisfies $g'(x^-)\ge g'(x^+)$ for every $x$.
\end{lemma}

\begin{lemma}
\label{lem:concave}
For every $v\in V$, the function $\varphi_v$ of \eqref{eq:message} is continuous on $\mathbb R$. For every $v\ne r$ and every $t\in\mathbb R$, the minimum in \eqref{eq:message} is attained, and $\mu_v:\mathbb R\to\mathbb R$ is concave.
\end{lemma}

\begin{proof}
Induction on the height of $v$. Let $v\in V$ and assume the lemma for all $j\in\ch(v)$, a vacuous assumption if $v$ is a leaf. Each $\mu_j$ is then continuous by Lemma~\ref{lem:concavejumps}, hence so is $\varphi_v$, and the minimum in \eqref{eq:message} is attained by compactness of $[0,1]$. Moreover, $\mu_v$ is the pointwise minimum of the affine functions $t\mapsto\varphi_v(x)+s_vtx$, $x\in[0,1]$, and it is finite, hence concave.
\end{proof}

The next lemma justifies the recursion \eqref{eq:message}.

\begin{lemma}
\label{lem:dp}
For every $v\ne r$ and every $t\in\mathbb R$,
\[
\mu_v(t) \;=\; \min\big\{F_v(y;t)\ :\ y\in[0,1]^{T_v}\big\},
\]
and $\min_{x\in[0,1]^n} f = \min_{x\in[0,1]}\varphi_r(x)$.
\end{lemma}

\begin{proof}
Induction on the height of $v$. By \eqref{eq:Fsplit}, minimizing $F_v(y;t)$ over $y\in[0,1]^{T_v}$ amounts to fixing $y_v=x\in[0,1]$ and minimizing each $F_j(\cdot\,;x)$ over $[0,1]^{T_j}$ separately, the blocks $T_j$, $j\in \ch(v)$, being disjoint; by the inductive hypothesis the $j$-th of these minima is $\mu_j(x)$, so that the whole minimum equals $\min_{x\in[0,1]}[\varphi_v(x)+s_vtx]=\mu_v(t)$ by \eqref{eq:message}. The root identity is this computation for $v=r$, where $s_r=0$ and $F_r(y;0)=f(y)$.
\end{proof}

\subsection{The convex envelope of a piecewise quadratic function with concave kinks}
\label{sec:envelope}

This subsection is self-contained univariate convex analysis, with no reference to the tree: the results below concern the class $\mathcal Q$ of pw-quadratic functions $\varphi:[0,1]\to\mathbb R$ all of whose kinks are concave, and the symbols $\varphi$ and $\mu$ no longer refer to the specific $\varphi_v$ and $\mu_v$ of \eqref{eq:message}.

We recall that the \emph{convex envelope} $\env\varphi$ of a continuous function $\varphi:[0,1]\to\mathbb R$ is the greatest convex function $\le\varphi$ on $[0,1]$, and we say that $x_0\in[0,1]$ is a \emph{contact point} of $\varphi$ if $(\env\varphi)(x_0)=\varphi(x_0)$; the set of all contact points is the \emph{contact set} of $\varphi$. The \emph{conjugate} of a continuous function $g:[0,1]\to\mathbb R$ is $g^*(\sigma):=\max_{x\in[0,1]}\big(\sigma x-g(x)\big)$, $\sigma\in\mathbb R$, which is the Fenchel conjugate of $g$ extended by $+\infty$ outside $[0,1]$. These notions, and all the results of this subsection, hold verbatim on an arbitrary compact interval $[\lambda,\rho]$ with $\lambda<\rho$: the affine bijection $x=\lambda+(\rho-\lambda)\hat x$ onto $[0,1]$ preserves convexity, arcs, kinks and their concavity. 

We first collect the basic properties of the convex envelope and describe its contact set with $\varphi$; we then use this description for the function obtained by minimizing $\varphi(x)+stx$ over $[0,1]$, which depends on $\varphi$ only through $\env\varphi$ (Lemma~\ref{lem:conjugate}). We conclude with two facts used in Section~\ref{sec:computation}: where the minimum of a function of $\mathcal Q$ is attained (Lemma~\ref{lem:candidates}), and the replacement of strictly concave arcs by their chords (Lemma~\ref{lem:chord}).

\begin{lemma}
\label{lem:carrier}
Let $\varphi:[0,1]\to\mathbb R$ be continuous and $H:=\env\varphi$. Then $H$ is continuous on $[0,1]$, $H(0)=\varphi(0)$, $H(1)=\varphi(1)$, and
\begin{equation}
\label{eq:twopoint}
H(x)=\min\big\{\omega\varphi(x_1)+(1-\omega)\varphi(x_2)\ :\ \omega\in[0,1],\ x_1,x_2\in[0,1],\ \omega x_1+(1-\omega)x_2=x\big\}.
\end{equation}
On every maximal open interval on which $H<\varphi$, $H$ is affine, and $H=\varphi$ at the endpoints of that interval.
\end{lemma}

\begin{proof}
Extend $\varphi$ by $+\infty$ outside $[0,1]$, so that $H$ is the convex hull of $\varphi$ in the sense of~\cite{Roc70}. By Corollary~17.1.5 in~\cite{Roc70}, applied with $n=1$, $H(x)$ is the infimum of $\omega\varphi(x_1)+(1-\omega)\varphi(x_2)$ over all $\omega\in[0,1]$ and $x_1,x_2\in[0,1]$ with $\omega x_1+(1-\omega)x_2=x$, and the infimum is attained since $\varphi$ is continuous on the compact interval $[0,1]$; this is \eqref{eq:twopoint}. Since $\varphi$ is real-valued and continuous on the compact set $[0,1]$, Corollary~17.2.1 in~\cite{Roc70} shows that $H$ is a closed proper convex function, and $[0,1]$ is a segment, hence a locally simplicial subset of the effective domain of $H$; by Theorem~10.2 in~\cite{Roc70}, $H$ is continuous relative to $[0,1]$. In \eqref{eq:twopoint} with $x=0$, every point carrying a positive weight equals $0$, since $0$ is an extreme point of $[0,1]$; hence $H(0)=\varphi(0)$, and similarly $H(1)=\varphi(1)$.

Now let $(\alpha,\beta)$ be a maximal open interval with $H<\varphi$, and $x_0\in(\alpha,\beta)$. Let $(\omega,x_1,x_2)$ attain the minimum in \eqref{eq:twopoint} at $x=x_0$, labelled so that $x_1\le x_2$; then $\omega\in(0,1)$ and $x_1<x_0<x_2$, since otherwise $H(x_0)=\varphi(x_0)$. Let $\ell$ be the affine function with $\ell(x_i)=\varphi(x_i)$, $i=1,2$; then $H(x_0)=\ell(x_0)$. By \eqref{eq:twopoint}, for every $x\in[x_1,x_2]$,
$H(x)\le\ell(x)$ (take the combination of $x_1,x_2$ realizing $x$), while
$\ell(x_0)=H(x_0)\le\omega H(x_1)+(1-\omega)H(x_2)\le\omega\varphi(x_1)+(1-\omega)\varphi(x_2)=\ell(x_0)$
forces $H(x_1)=\varphi(x_1)$, $H(x_2)=\varphi(x_2)$, and, by the equality case of convexity at the interior point $x_0$, $H=\ell$ on $[x_1,x_2]$. In particular $x_1\le\alpha$ and $x_2\ge\beta$, as $x_1,x_2$ are contact points, so $H$ is affine on $[\alpha,\beta]$; and $H=\varphi$ at $\alpha$ and $\beta$ by the maximality of $(\alpha,\beta)$ and the continuity of $H$ and $\varphi$.
\end{proof}

By Lemma~\ref{lem:carrier}, for continuous $\varphi$ the contact set is closed, being the zero set of the continuous function $\varphi-\env\varphi$, and its complement in $[0,1]$ is an at most countable union of open intervals on each of which $\env\varphi$ is affine; we call the closures of these intervals \emph{bridges}. If moreover $\varphi$ is pw-quadratic, we call \emph{contact piece} any nonempty intersection of a maximal interval of the contact set of $\varphi$ with an arc of $\varphi$; a contact piece is \emph{degenerate} if it is a single point.

\begin{lemma}
\label{lem:bridge}
Let $\lambda<\xi<\rho$, let $g:[\lambda,\rho]\to\mathbb R$ be continuous, and set $H_1:=\env\big(g|_{[\lambda,\xi]}\big)$, $H_2:=\env\big(g|_{[\xi,\rho]}\big)$ and $H:=\env(g)$. Then there exist $\eta\in[\lambda,\xi]$ and $\zeta\in[\xi,\rho]$ such that $H=H_1$ on $[\lambda,\eta]$, $H$ is affine on $[\eta,\zeta]$, and $H=H_2$ on $[\zeta,\rho]$. Moreover, let $w\in(\lambda,\xi]$ be such that $H_1'(w^-)$ is finite, and let $\ell_0$ be the line through $(w,H_1(w))$ with slope $H_1'(w^-)$. Then $H=H_1$ on $[\lambda,w]$ if and only if $\ell_0\le g$ on $[\xi,\rho]$.
\end{lemma}

\begin{proof}
By Lemma~\ref{lem:carrier} we have $H_1(\xi)=g(\xi)=H_2(\xi)$, so the function $\tilde H$ equal to $H_1$ on $[\lambda,\xi]$ and to $H_2$ on $[\xi,\rho]$ is well defined and continuous. We claim $\env\tilde H=H$. Indeed $\tilde H\le g$, so $\env\tilde H\le H$; conversely $H$ is convex with $H\le g$, so its restriction to $[\lambda,\xi]$ is a convex function $\le g|_{[\lambda,\xi]}$, and therefore $H\le H_1$ there, and similarly $H\le H_2$ on $[\xi,\rho]$, that is $H\le\tilde H$, so $H\le\env\tilde H$.

Apply now Lemma~\ref{lem:carrier} to $\tilde H$: on every maximal open interval on which $H<\tilde H$, the function $H$ is affine and agrees with $\tilde H$ at the endpoints. Suppose such a component $(\alpha_1,\beta_1)$ were contained in $[\lambda,\xi]$. Then $H$ is the chord of $H_1$ over $[\alpha_1,\beta_1]$, and since $H_1$ is convex that chord is at least $H_1=\tilde H$ on $[\alpha_1,\beta_1]$, contradicting $H<\tilde H$ there. The same argument excludes components contained in $[\xi,\rho]$. Hence every such component contains $\xi$, so there is at most one; taking $[\eta,\zeta]$ to be its closure, or $\eta=\zeta=\xi$ if there is none, proves the first part.

For the second part, suppose first that $\ell_0\le g$ on $[\xi,\rho]$. By convexity, $\ell_0\le H_1$ on $[\lambda,\xi]$, and $\ell_0\le H_2$ on $[\xi,\rho]$ since $\ell_0$ is affine and $\ell_0\le g$ there; so $\ell_0\le\tilde H$ and hence $\ell_0\le H$; thus $H(w)=H_1(w)$. If $w>\eta$, then $H$ is affine on $[\eta,w]$ and agrees with the convex function $H_1\ge H$ at $\eta$ and $w$, so $H=H_1$ on $[\eta,w]$; in either case $H=H_1$ on $[\lambda,w]$. Conversely, suppose that $H=H_1$ on $[\lambda,w]$, and let $\ell$ be a supporting line of $H$ at $w$, which exists as $w<\rho$. Its slope is at least $H'(w^-)=H_1'(w^-)$, and $\ell(w)=\ell_0(w)$, so $\ell_0\le\ell\le H\le g$ on $[w,\rho]\supseteq[\xi,\rho]$.
\end{proof}

\begin{lemma}
\label{lem:tangent}
Let $\varphi\in\mathcal Q$, let $H:=\env\varphi$, and let $x_0\in(0,1)$ be a contact point of $\varphi$. Then
\begin{equation}
\label{eq:contactslopes}
\varphi'(x_0^-)\ \le\ m\ \le\ \varphi'(x_0^+)
\qquad\text{for every } m\in\partial H(x_0).
\end{equation}
In particular, $\varphi$ and $H$ are differentiable at $x_0$, with $H'(x_0)=\varphi'(x_0)$.
\end{lemma}

\begin{proof}
The function $H$ is convex and finite on $[0,1]$ and $x_0\in(0,1)$, so $\partial H(x_0)\ne\emptyset$ by Theorem~23.4 in~\cite{Roc70}. Let $m\in\partial H(x_0)$ and let $\ell(x):=H(x_0)+m(x-x_0)$ be the corresponding supporting line, so that $\ell\le H\le\varphi$, and $\ell(x_0)=H(x_0)=\varphi(x_0)$ because $x_0$ is a contact point. Hence $\varphi(x)-\varphi(x_0)\ge m(x-x_0)$ for every $x\in[0,1]$; dividing by $x-x_0$ and letting $x\to x_0$ on either side yields \eqref{eq:contactslopes}, the one-sided derivatives existing because $\varphi$ is pw-quadratic.

Since $\partial H(x_0)\ne\emptyset$, \eqref{eq:contactslopes} forces $\varphi'(x_0^-)\le\varphi'(x_0^+)$. A kink of a function of $\mathcal Q$ is concave, that is, it satisfies $\varphi'(x_0^-)>\varphi'(x_0^+)$; so $x_0$ is not a kink, and $\varphi$ is differentiable at $x_0$. The two bounds in \eqref{eq:contactslopes} then coincide, so $\partial H(x_0)$ is the singleton $\{\varphi'(x_0)\}$; by Theorem~25.1 in~\cite{Roc70}, $H$ is differentiable at $x_0$ with $H'(x_0)=\varphi'(x_0)$.
\end{proof}

\begin{lemma}
\label{lem:noconcave}
Let $\varphi\in\mathcal Q$. Then no contact point of $\varphi$ lies in the interior of a strictly concave arc.
\end{lemma}

\begin{proof}
Write $H:=\env\varphi$ and suppose, for a contradiction, that some contact point $x_0$ lies in the interior of an arc on which $\varphi(x)=ax^2+bx+d$ with $a<0$; let $N\subseteq(0,1)$ be a neighborhood of $x_0$ contained in that arc. By Lemma~\ref{lem:tangent}, $m:=\varphi'(x_0)$ is a subgradient of $H$ at $x_0$, so $H(x)\ge\varphi(x_0)+m(x-x_0)$ for all $x\in[0,1]$, as $H(x_0)=\varphi(x_0)$. For $x\in N\setminus\{x_0\}$, since $\varphi$ is quadratic on $N$,
$\varphi(x)-H(x)\le\varphi(x)-\varphi(x_0)-m(x-x_0)=a(x-x_0)^2<0$, contradicting $H\le\varphi$.
\end{proof}

\begin{lemma}
\label{lem:onecontact}
Let $\varphi\in\mathcal Q$, $H:=\env\varphi$, and let $A$ be an arc of $\varphi$ with nonnegative quadratic coefficient. If $x_1<x_2$ are contact points of $\varphi$ in $A$, then $H=\varphi$ on $[x_1,x_2]$. Consequently, the contact set of $\varphi$ meets every arc with nonnegative quadratic coefficient in a single (possibly empty or degenerate) closed interval.
\end{lemma}

\begin{proof}
Suppose, for a contradiction, that $H(x_0)<\varphi(x_0)$ for some $x_0\in(x_1,x_2)$, and let $(\alpha,\beta)$ be the maximal open interval containing $x_0$ on which $H<\varphi$. Since $x_1$ and $x_2$ are contact points, neither belongs to $(\alpha,\beta)$, so that
\[
x_1\le\alpha<x_0<\beta\le x_2 ,
\]
and in particular $[\alpha,\beta]\subseteq[x_1,x_2]\subseteq A$. By Lemma~\ref{lem:carrier}, $H$ is affine on $[\alpha,\beta]$ and $H=\varphi$ at $\alpha$ and at $\beta$; that is, $H$ coincides on $[\alpha,\beta]$ with the chord of $\varphi$ over that interval. But $\varphi$ is convex on $A$, so its chord over $[\alpha,\beta]$ is at least $\varphi$ there. Hence $H\ge\varphi$ on $(\alpha,\beta)$, contradicting $H<\varphi$ there. The final sentence follows, since the contact set is closed and, by the first assertion, its intersection with an arc of nonnegative quadratic coefficient is convex.
\end{proof}

\begin{lemma}
\label{lem:conjugate}
Let $\varphi\in\mathcal Q$ have $K$ arcs, let $H:=\env\varphi$, let $s\in\mathbb R\setminus\{0\}$, and define
\[
\mu(t):=\min_{x\in[0,1]}\big[\varphi(x)+stx\big],\qquad t\in\mathbb R.
\]
Then $\mu$ is a concave pw-quadratic function on $\mathbb R$ with at most $K+2$ arcs. Specifically, $\mu(t)=\varphi(0)$ for $-st\le H'(0^+)$, $\mu(t)=\varphi(1)+st$ for $-st\ge H'(1^-)$, and, for each nondegenerate contact piece $[w_1,w_2]$ inside a strictly convex arc on which $\varphi(x)=ax^2+bx+d$,
\begin{equation}
\label{eq:conjformula}
\mu(t) = -\frac{(st+b)^2}{4a} + d \qquad\text{for all $t$ with } -st\in[\,2aw_1+b,\ 2aw_2+b\,].
\end{equation}
These ranges of $t$ cover $\mathbb R$ except for finitely many points, at most one for each maximal nondegenerate interval on which $H$ is affine.
\end{lemma}

\begin{proof}
We first observe that
\[
\mu(t)=-\max_{x\in[0,1]}\big[(-st)x-\varphi(x)\big]=-\varphi^*(-st)=-H^*(-st).
\]
The last equality holds because an affine function lies below $\varphi$ if and only if it lies below $H$: extending $\varphi$ by $+\infty$ outside $[0,1]$, its closed convex hull is $H$, and Corollary~12.1.1 in~\cite{Roc70} applies. Hence $\varphi^*=H^*$, and it suffices to describe $H^*$.

By Lemma~\ref{lem:carrier} and Lemmas~\ref{lem:tangent}--\ref{lem:onecontact}, $H$ is convex and pw-quadratic: it coincides with $\varphi$ on the contact set (which meets each arc with nonnegative quadratic coefficient in at most one interval and avoids the interiors of strictly concave arcs and all kinks) and is affine on the bridges in between. By Lemma~\ref{lem:tangent}, every contact point in $(0,1)$ is a point at which $\varphi$ is differentiable and $H$ is differentiable with $H'=\varphi'$; on bridges $H$ is affine; hence $H$ is differentiable on $(0,1)$ and $H'$ is continuous and nondecreasing, strictly increasing on contact pieces inside strictly convex arcs (where $H'(x)=2ax+b$) and constant on intervals on which $H$ is affine.

Fix $\sigma\in\big(H'(0^+),H'(1^-)\big)$. Since $H'$ is continuous and nondecreasing on $(0,1)$, $\sigma=H'(x)$ for some $x\in(0,1)$, and $x$ lies in a nondegenerate contact piece inside a strictly convex arc or in a maximal nondegenerate interval on which $H$ is affine, as $H$ is affine on bridges and on contact pieces inside affine arcs, and an isolated contact point in $(0,1)$ lies between two bridges, which have the same slope by Lemma~\ref{lem:tangent}. The maximum of the concave function $x\mapsto\sigma x-H(x)$ is attained exactly on $\{x: \sigma\in\partial H(x)\}=\{x: H'(x)=\sigma\}$. If $\sigma$ lies in the interior of the slope range $[2aw_1+b,\,2aw_2+b]$ of a contact piece inside a strictly convex arc, the maximizer $x=(\sigma-b)/2a$ is unique, and since $H=\varphi$ there,
$H^*(\sigma)=\sigma x-(ax^2+bx+d)=\frac{(\sigma-b)^2}{4a}-d$:
one quadratic arc of $H^*$ per such contact piece. If $\sigma$ is the slope of a maximal nondegenerate interval on which $H$ is affine, the maximizer set is that interval, which equals $\partial H^*(\sigma)$ by Corollary~23.5.1 in~\cite{Roc70}; thus $\sigma$ contributes no arc of $H^*$, and is a kink of $H^*$. Distinct such intervals and distinct contact pieces occupy slope ranges with disjoint interiors, consecutively ordered, since $H'$ is nondecreasing; so the arcs of $H^*$ on $\big(H'(0^+),H'(1^-)\big)$ are exactly one per nondegenerate strictly convex contact piece. For $\sigma\le H'(0^+)$, $x=0$ is a maximizer and $H^*(\sigma)=-H(0)=-\varphi(0)$; for $\sigma\ge H'(1^-)$, $x=1$ is a maximizer and $H^*(\sigma)=\sigma-H(1)=\sigma-\varphi(1)$ (using Lemma~\ref{lem:carrier} for the endpoint values): the two unbounded affine arcs.

Thus $H^*$ is convex pw-quadratic, with at most two arcs more than the number of strictly convex arcs of $\varphi$, each of which contains at most one contact piece by Lemma~\ref{lem:onecontact}, hence with at most $K+2$ arcs. Substituting $\sigma=-st$ and negating yields the statement, including \eqref{eq:conjformula}; concavity of $\mu$ is immediate since $t\mapsto H^*(-st)$ is convex.
\end{proof}

\begin{lemma}
\label{lem:candidates}
Let $\varphi\in\mathcal Q$ have arcs $[\xi_{j-1},\xi_j]$, $j\in[K]$, where $0=\xi_0<\xi_1<\dots<\xi_K=1$, with $\varphi(x)=a_jx^2+b_jx+d_j$ on the $j$-th arc. Then $\varphi$ attains its minimum over $[0,1]$ at $0$, at $1$, or at $-b_j/(2a_j)$ for some $j$ with $a_j>0$ and $-b_j/(2a_j)\in[\xi_{j-1},\xi_j]$.
\end{lemma}

\begin{proof}
Let $\mathcal M$ be the set of minimizers of $\varphi$ over $[0,1]$, which is nonempty and compact by continuity, and let $x^\star:=\max\mathcal M$. If $x^\star\in\{0,1\}$ we are done, so assume $x^\star\in(0,1)$. Then $x^\star$ is not a kink of $\varphi$: all kinks are concave, that is $\varphi'(x^{\star-})>\varphi'(x^{\star+})$, whereas minimality in the interior gives $\varphi'(x^{\star-})\le0\le\varphi'(x^{\star+})$. Hence $\varphi$ is differentiable at $x^\star$ and $\varphi'(x^\star)=0$. Let $[\xi_{j-1},\xi_j]$ be the arc immediately to the right of $x^\star$, that is, the arc with $\xi_{j-1}=x^\star$ if $x^\star$ is a breakpoint and the arc containing $x^\star$ in its interior otherwise; thus $2a_jx^\star+b_j=0$. If $a_j<0$ then $\varphi$ is strictly decreasing immediately to the right of $x^\star$, contradicting minimality; if $a_j=0$ then $b_j=0$ and $\varphi$ is constant on that arc, so every point of $[x^\star,\xi_j]$ is a minimizer, contradicting the maximality of $x^\star$. Hence $a_j>0$ and $x^\star=-b_j/(2a_j)\in[\xi_{j-1},\xi_j]$.
\end{proof}

The last result of this subsection allows us to dispense with strictly concave arcs when computing convex envelopes.

\begin{lemma}
\label{lem:chord}
Let $\varphi\in\mathcal Q$ have arcs $[\xi_{j-1},\xi_j]$, $j\in[K]$, where $0=\xi_0<\xi_1<\dots<\xi_K=1$, and let $\hat\varphi$ be obtained from $\varphi$ by replacing every strictly concave arc by its chord. Then $\hat\varphi\in\mathcal Q$, $\hat\varphi$ is strictly convex or affine on each $[\xi_{j-1},\xi_j]$, $\hat\varphi(\xi_j)=\varphi(\xi_j)$ for every $j$, and $\env(\hat\varphi|_{[0,\xi_j]})=\env(\varphi|_{[0,\xi_j]})$ for every $j\in[K]$.
\end{lemma}

\begin{proof}
Clearly $\hat\varphi$ is continuous, agrees with $\varphi$ at every $\xi_j$, and is strictly convex or affine on each $[\xi_{j-1},\xi_j]$; its breakpoints are among the $\xi_j$. On a strictly concave arc $[\alpha,\beta]$, the slope $\gamma$ of the chord satisfies $\varphi'(\alpha^+)>\gamma>\varphi'(\beta^-)$. Hence at every breakpoint $\xi$ of $\varphi$ the left derivative of $\hat\varphi$ is at least $\varphi'(\xi^-)$ and the right derivative is at most $\varphi'(\xi^+)$; as $\varphi'(\xi^-)\ge\varphi'(\xi^+)$, every kink of $\hat\varphi$ is concave, and $\hat\varphi\in\mathcal Q$. Finally fix $j$. Since $\hat\varphi\le\varphi$, we have $\env(\hat\varphi|_{[0,\xi_j]})\le\env(\varphi|_{[0,\xi_j]})$. Conversely, $\env(\varphi|_{[0,\xi_j]})$ is convex and at most $\varphi=\hat\varphi$ at the endpoints of every arc, hence at most every chord, and so at most $\hat\varphi$ on $[0,\xi_j]$; therefore $\env(\varphi|_{[0,\xi_j]})\le\env(\hat\varphi|_{[0,\xi_j]})$.
\end{proof}

\subsection{Structure of the value functions}
\label{sec:arithmetic}

We now apply the results of Section~\ref{sec:envelope} to the value functions of \eqref{eq:message}. We write $K_v$ for the number of arcs of $\varphi_v|_{[0,1]}$.

\begin{proposition}
\label{prop:master}
For every $v\in V$, the restriction of $\varphi_v$ to $[0,1]$ is pw-quadratic with $K_v\le 2|T_v|-1$ arcs and all its kinks are concave; that is, $\varphi_v|_{[0,1]}\in\mathcal Q$. Moreover, for every $v\ne r$, the value function $\mu_v$ is a concave pw-quadratic function on $\mathbb R$ with at most $2|T_v|+1$ arcs.
\end{proposition}

\begin{proof}
Induction on the height of $v$. Let $v\in V$ and assume the proposition for all $j\in\ch(v)$, a vacuous assumption if $v$ is a leaf. Each $\mu_j$ is then pw-quadratic on $\mathbb R$ with at most $2|T_j|$ breakpoints, so $\varphi_v|_{[0,1]}$ is pw-quadratic and its breakpoints are among the breakpoints of the $\mu_j$ lying in $(0,1)$, of which there are at most $\sum_{j\in\ch(v)}2|T_j|=2(|T_v|-1)$; hence $K_v\le2|T_v|-1$. At every breakpoint, the jump of $\varphi_v'$ is the sum of the jumps of the $\mu_j'$, each nonpositive by Lemmas~\ref{lem:concave} and~\ref{lem:concavejumps}; so every kink of $\varphi_v$ is concave, and $\varphi_v|_{[0,1]}\in\mathcal Q$.

Let now $v\ne r$. By \eqref{eq:message}, $\mu_v$ is the function $\mu$ of Lemma~\ref{lem:conjugate} for $(\varphi,s)=(\varphi_v|_{[0,1]},s_v)$, where $s_v\ne0$. Hence $\mu_v$ is a concave pw-quadratic function with at most $K_v+2\le2|T_v|+1$ arcs, which completes the induction.
\end{proof}

\begin{remark}
\label{rem:tight}
The bounds of Proposition~\ref{prop:master} are attained, even when $Q\succ0$. For $n\ge2$ and $w_i:=4^{i-1}$, let the objective of Problem~\eqref{eq:QP} be $\sum_{i=1}^{n-1}w_i(x_i-3x_{i+1}+1)^2+w_n(x_n-\frac12)^2$ minus its constant term, so that the interaction graph is the path $1,2,\dots,n$, which we root at $n$. It can be checked that every $\varphi_i$ is strongly convex and every breakpoint of $\mu_i$ lies in $(0,1)$, and hence that $K_i=2i-1$ for every $i\in[n]$ and $\mu_i$ has $2i+1$ arcs for every $i<n$. In particular, $\sum_{v\in V}K_v=n^2$.
\end{remark}


We next show that all arcs produced by the recursion~\eqref{eq:message} have rational coefficients.

\begin{lemma}
\label{lem:rational}
Every arc of every $\mu_v$, $v\ne r$, and of every $\varphi_v|_{[0,1]}$ has rational coefficients. Every breakpoint $\tau$ of $\mu_v$ is a root of $\pi_1-\pi_2$, where $\pi_1\ne\pi_2$ are the polynomials of the two arcs of $\mu_v$ meeting at $\tau$; this is a nonzero polynomial of degree at most two with rational coefficients, which we call the \emph{defining quadratic} of $\tau$. Every breakpoint of $\varphi_v|_{[0,1]}$ is a breakpoint of some $\mu_j$, $j\in\ch(v)$, and we take its defining quadratic to be that of this breakpoint of $\mu_j$.
\end{lemma}

\begin{proof}
By Proposition~\ref{prop:master} these functions are pw-quadratic. We argue by induction on the height of $v$. Each arc of $\varphi_v|_{[0,1]}$ coincides on a nondegenerate subinterval, and hence as a polynomial, with the sum of $q_{vv}x^2+c_vx$ and of one arc of each $\mu_j$, $j\in\ch(v)$, so it has rational coefficients by induction; and the breakpoints of $\varphi_v|_{[0,1]}$ are breakpoints of the $\mu_j$. By Lemma~\ref{lem:conjugate}, each arc of $\mu_v$ is $\varphi_v(0)$, or $\varphi_v(1)+s_vt$, or $-\frac{(s_vt+b)^2}{4a}+d$ with $(a,b,d)$ the coefficients of an arc of $\varphi_v$; since $\varphi_v(0)$ and $\varphi_v(1)$ are values of arcs of $\varphi_v$ at $0$ and $1$, and $s_v\in\Q$, it has rational coefficients. Let now $\tau$ be a breakpoint of $\mu_v$ and let $\pi_1,\pi_2$ be the polynomials of the two arcs of $\mu_v$ meeting at $\tau$. They have rational coefficients by what we just proved, and $\pi_1\ne\pi_2$ because arcs are maximal. By continuity of $\mu_v$, $\pi_1(\tau)=\pi_2(\tau)$, so $\tau$ is a root of the nonzero polynomial $\pi_1-\pi_2$, which has rational coefficients and degree at most two.
\end{proof}

We represent a real number $\tau$ of degree at most two over $\Q$ by rationals $\tau_0,\tau_1,D$ with $D\ge0$ and $\tau=\tau_0+\tau_1\sqrt D$, and call $(\tau_0,\tau_1,D)$ a \emph{representation} of $\tau$. If $\tau$ is a root of a nonzero polynomial $\pi$ of degree at most two with rational coefficients, a representation of $\tau$ is obtained from the coefficients of $\pi$ by the quadratic formula with $O(1)$ operations, once it is known which root of $\pi$ is $\tau$; this fixes the sign of $\tau_1$.

We next record an observation that holds for every box-constrained quadratic program, not only on trees; in it, the empty matrix is regarded as positive definite. It is implicit in the proof of NP membership of quadratic programming by Vavasis~\cite{Vav90}.

\begin{lemma}
\label{lem:pdface}
Let $Q\in\R^{n\times n}$ be symmetric and $c\in\R^n$. Then $\min\{x^\top Qx+c^\top x:x\in[0,1]^n\}$ has an optimal solution $x^\star$ such that $Q_S\succ0$, where $S:=\{i\in[n]:0<x^\star_i<1\}$.
\end{lemma}

\begin{proof}
Write $f(x):=x^\top Qx+c^\top x$. Among the optimal solutions, which exist by compactness, let $x^\star$ have the largest number of coordinates in $\{0,1\}$, and let $S$ be as in the statement. For every $d\in\R^n$ with $d_{[n]\setminus S}=0$, the point $x^\star+td$ is feasible for all $t$ in a neighborhood of $0$, and $f(x^\star+td)=f(x^\star)+t\,\nabla f(x^\star)^\top d+t^2\,d_S^\top Q_Sd_S$; optimality of $x^\star$ forces $\nabla_Sf(x^\star):=(\nabla f(x^\star))_S=0$ and $d_S^\top Q_Sd_S\ge0$, so $Q_S\succeq0$. If $Q_S$ were singular, pick $0\ne d_S\in\ker Q_S$, extended by zero outside $S$. Then $f(x^\star+td)=f(x^\star)$ for every $t$. Let $\bar t$ be the largest $t\ge0$ with $x^\star+td\in[0,1]^n$, which is finite as $d\ne0$; then $x^\star+\bar td$ is an optimal solution with more coordinates in $\{0,1\}$ than $x^\star$, a contradiction. Hence $Q_S\succ0$.
\end{proof}

We now give a closed-form description of the arcs of the value functions, which bypasses the recursion, along which lengths could a priori double at every level. In it, $S$ collects the coordinates that are left free, while the coordinates outside $S$ are fixed at the bounds prescribed by $\beta$.

\begin{lemma}
\label{lem:activeset}
For every $v\ne r$, every arc of $\mu_v$ coincides, as a polynomial in $t$, with the function
\begin{equation}
\label{eq:activeset}
t\ \longmapsto\ \min\big\{F_v(y;t)\ :\ y\in\mathbb R^{T_v},\ y_{T_v\setminus S}=\beta\big\}
\end{equation}
for some $S\subseteq T_v$ with $Q_S\succ0$ and some $\beta\in\{0,1\}^{T_v\setminus S}$.
\end{lemma}

\begin{proof}
For $S\subseteq T_v$ with $Q_S\succ0$ and $\beta\in\{0,1\}^{T_v\setminus S}$, denote by $\bar F_{S,\beta}$ the function \eqref{eq:activeset}. The restriction of $F_v(\cdot\,;t)$ to $\{y:y_{T_v\setminus S}=\beta\}$ is a strictly convex quadratic in $y_S$ whose coefficients are affine in $t$, so $\bar F_{S,\beta}$ is a polynomial in $t$ of degree at most two.

Now fix $t\in\R$. By Lemma~\ref{lem:pdface}, applied to the matrix $Q_{T_v}$ and the vector $c_{T_v}+s_vte_v$, the problem $\min\{F_v(y;t):y\in[0,1]^{T_v}\}$ has a minimizer $y^\star$ with $Q_S\succ0$ for $S:=\{u\in T_v:0<y^\star_u<1\}$; put $\beta:=y^\star_{T_v\setminus S}\in\{0,1\}^{T_v\setminus S}$. Since $y^\star$ is a minimizer and $y^\star_S$ lies in the interior of $[0,1]^S$, we have $\nabla_SF_v(y^\star;t)=0$; as $Q_S\succ0$, the point $y^\star$ minimizes $F_v(\cdot\,;t)$ over $\{y\in\R^{T_v}:y_{T_v\setminus S}=\beta\}$, so $\mu_v(t)=F_v(y^\star;t)=\bar F_{S,\beta}(t)$ by Lemma~\ref{lem:dp}.

Finally, let $\pi$ be an arc of $\mu_v$, a polynomial of degree at most two that coincides with $\mu_v$ on a nondegenerate interval $I$. As there are finitely many pairs $(S,\beta)$ and infinitely many $t\in I$, some pair satisfies $\pi(t)=\bar F_{S,\beta}(t)$ for infinitely many $t$, and then $\pi=\bar F_{S,\beta}$ as polynomials.
\end{proof}

We can now bound the lengths of all the numbers that describe the value functions.

\begin{lemma}
\label{lem:bits}
All arc coefficients of all $\mu_v$ and $\varphi_v|_{[0,1]}$, and the coefficients of the defining quadratics and the representations of their breakpoints, are rationals of length polynomial in $\L$.
\end{lemma}

\begin{proof}
All the quantities in question are rational by Lemma~\ref{lem:rational}; it remains to bound their lengths. We use that sums and products of polynomially many rationals of length polynomial in $\L$ have length polynomial in $\L$.

Consider first an arc of a value function $\mu_v$ with $v\ne r$, and let $S$ and $\beta$ be as furnished by Lemma~\ref{lem:activeset}. Fixing $y_{T_v\setminus S}=\beta$ in \eqref{eq:Fv} and separating the terms according to whether they involve the free block or not, we obtain
\[
F_v(y;t)=y_S^\top Q_S\,y_S+\big(\theta+t\psi\big)^\top y_S+\big(\gamma_0+\gamma_1t\big)
\qquad\text{for every } y\in\mathbb R^{T_v} \text{ with } y_{T_v\setminus S}=\beta,
\]
where, for $u\in S$,
\[
\theta_u=c_u+\sum_{\substack{w\in T_v\setminus S\\ \{u,w\}\in E}}2q_{uw}\,\beta_w ,
\qquad
\psi=\begin{cases} s_v\,(e_v)_S & v\in S,\\ 0 & v\notin S,\end{cases}
\]
and $\gamma_0,\gamma_1$ are obtained from the data and from $\beta$ in the same manner. Each $\beta_w$ lies in $\{0,1\}$, so each entry of $\theta$, $\psi$, $\gamma_0$ and $\gamma_1$ is, up to a factor two, a sum of at most $2n^2$ entries of $Q$ and $c$, and therefore has length polynomial in $\L$.

Since $Q_S\succ0$, the minimum in \eqref{eq:activeset} is attained at $y_S=-\tfrac12Q_S^{-1}(\theta+t\psi)$, and the arc formula is
\begin{equation}
\label{eq:arcformula}
\min\big\{F_v(y;t)\ :\ y\in\mathbb R^{T_v},\ y_{T_v\setminus S}=\beta\big\}
=\gamma_0+\gamma_1t-\tfrac14\big(\theta+t\psi\big)^\top Q_S^{-1}\big(\theta+t\psi\big).
\end{equation}
By Corollary~3.2a in~\cite{SchBookIP}, the entries of $Q_S^{-1}$ have length polynomial in the length of $Q_S$, a submatrix of $Q$, and hence polynomial in $\L$. Expanding \eqref{eq:arcformula}, each coefficient of the arc is a sum of at most $n^2+1$ products of at most three such rationals, and therefore has length polynomial in $\L$. The same bound holds for the arcs of $\varphi_v$, since by \eqref{eq:message} each of them is a sum of the quadratic $q_{vv}x^2+c_vx$ and of at most $n$ arc formulas of value functions. Finally, by Lemma~\ref{lem:rational} the defining quadratic of a breakpoint is the difference of two arc polynomials, so its coefficients have length polynomial in $\L$ as well, and so do the rationals of a representation, which are obtained from them by $O(1)$ operations.
\end{proof}

\subsection{Exact computation of the value functions}
\label{sec:computation}

Computing the value functions requires comparing their breakpoints, and also auxiliary numbers, obtained from them by one further square root, that arise while a single value function is computed (Lemma~\ref{lem:envelope}). The next lemma covers both; it is folklore, and we include a proof for completeness.

\begin{lemma}
\label{lem:compare}
Let $k$ and $d$ be fixed positive integers. Let $E_1\in\Q$ and, for $i=2,\dots,k$, let $E_i$ be a polynomial of total degree at most $d$ with rational coefficients in $\sqrt{E_1},\dots,\sqrt{E_{i-1}}$, where every $E_i$ is a nonnegative real number. Then the sign of a polynomial of total degree at most $d$ with rational coefficients in $\sqrt{E_1},\dots,\sqrt{E_k}$ can be determined with $O(1)$ arithmetic operations and comparisons on rational numbers, starting from $E_1$ and the coefficients of the given polynomials.
\end{lemma}

\begin{proof}
Using $\sqrt{E_i}^{\,2}=E_i$, every such polynomial $\alpha$ can be written as $\alpha=P+R\sqrt{E_k}$, where $P$ and $R$ are polynomials with rational coefficients in $\sqrt{E_1},\dots,\sqrt{E_{k-1}}$; recursively, $\alpha$ is represented by $2^k$ rational numbers, and sums and products of such numbers cost $O(1)$ operations. As all the given polynomials have total degree at most $d$, these representations are obtained with $O(1)$ operations. We argue by induction on $k$, the case $k=0$ being a comparison of a rational number with zero. Let $k\ge1$ and determine the signs of $P$ and $R$ by induction. If $R=0$, or if $P$ and $R$ are both positive or both negative, the sign of $\alpha$ is that of $P$. If $P=0\ne R$, the sign of $\alpha$ is the product of the signs of $R$ and $E_k$, and the sign of $E_k$ is determined by induction. Otherwise $P$ and $R$ have opposite signs, and
\[
\operatorname{sign}\big(P+R\sqrt{E_k}\big)=\operatorname{sign}(P)\cdot\operatorname{sign}\big(P^2-R^2E_k\big),
\]
since for $P>0>R$ we have $P+R\sqrt{E_k}>0$ if and only if $P>-R\sqrt{E_k}$, that is $P^2>R^2E_k$, and symmetrically for $P<0<R$. As $E_k$ is a polynomial in $\sqrt{E_1},\dots,\sqrt{E_{k-1}}$, so is $P^2-R^2E_k$, and its sign is determined by induction. Altogether the number of operations is bounded by a constant depending only on $k$ and $d$.
\end{proof}

We compute a value function from the corresponding function $\varphi$ in two steps: first the convex envelope of $\varphi$, then the value function itself. For $\varphi\in\mathcal Q$ with arcs $[\xi_{j-1},\xi_j]$, let $\hat\varphi$ be as in Lemma~\ref{lem:chord}, so that $\env\varphi=\env\hat\varphi$. By Lemmas~\ref{lem:carrier} and~\ref{lem:tangent}--\ref{lem:onecontact}, applied to $\hat\varphi$, this envelope is determined by the contact pieces of $\hat\varphi$, consecutive ones being either adjacent or joined by a bridge. We call \emph{envelope list} of $\varphi$ the list of these contact pieces in increasing order, each given by the index $j$ of the arc containing it and by its two endpoints; the \emph{carrier} of such a contact piece is the polynomial of degree at most two that coincides with $\hat\varphi$ on the arc $j$, that is, the polynomial of the arc $j$ of $\varphi$, or its chord if that arc is strictly concave. The sweep in the proof below follows the scheme of the linear-time convex envelope algorithm of Gardiner and Lucet~\cite{GarLuc10} for univariate piecewise linear-quadratic functions, which is designed for real arithmetic. What the next lemma adds is that, for functions of $\mathcal Q$ with rational data, the sweep can be organized so that every endpoint it computes is obtained from the data with at most two nested square roots; this is what allows it to be carried out exactly, with $O(1)$ operations on rational numbers per step.

\begin{lemma}
\label{lem:envelope}
Let $\varphi\in\mathcal Q$ have $K$ arcs with rational coefficients and breakpoints given by representations. Then the envelope list of $\varphi$ can be computed with $O(K)$ arithmetic operations and comparisons on rational numbers. Each endpoint in it is given as a polynomial with rational coefficients in at most two square roots, nested to depth at most two.
\end{lemma}

\begin{proof}
Throughout, every comparison performed is the determination of the sign of a polynomial of bounded degree with rational coefficients in a bounded number of square roots, nested to depth at most two, of the kind described in the next paragraph; each of them costs $O(1)$ operations and comparisons on rational numbers by Lemma~\ref{lem:compare}, and we do not repeat this below. Some of these tests involve quotients, such as slopes of lines through two stored points; their signs are obtained from those of numerators and denominators.

Write $0=\xi_0<\dots<\xi_K=1$ for the endpoints of the arcs of $\varphi$, with representations $\xi_i=\xi_{i,0}+\xi_{i,1}\sqrt{D_i}$, and $(a_j,b_j,d_j)\in\Q^3$ for the coefficients of its $j$-th arc. We work with $\hat\varphi$ throughout, which is strictly convex or affine on each arc, and represent the convex envelopes $H_j:=\env(\varphi|_{[0,\xi_j]})=\env(\hat\varphi|_{[0,\xi_j]})$ computed below by lists of contact pieces of $\hat\varphi|_{[0,\xi_j]}$, as in the envelope list. Every carrier is determined by the index of its arc, so no new coefficients are stored; the chord of a strictly concave arc $j$ is used only through the points $(\xi_{j-1},\varphi(\xi_{j-1}))$ and $(\xi_j,\varphi(\xi_j))$. The endpoints are of three kinds: breakpoints $\xi_i$ of $\varphi$; abscissae of the tangency points of a common tangent of two parabolas with rational coefficients, which are again of the form $\tau_0+\tau_1\sqrt D$ with $\tau_0,\tau_1,D\in\Q$; and abscissae $\eta$ of the tangency point of the tangent from a point $(\xi_i,\varphi(\xi_i))$ to a parabola $y=ax^2+bx+d$ with rational coefficients, which satisfy $a\eta^2-2a\xi_i\eta+\varphi(\xi_i)-b\xi_i-d=0$ and are therefore of the form $\eta=\xi_i\pm\sqrt{\Delta}$ with $\Delta$ a polynomial with rational coefficients in $\xi_i$, hence in $\sqrt{D_i}$. We store every endpoint as such an expression. It is then a polynomial with rational coefficients in at most two square roots nested to depth at most two, and every test below is a polynomial in the square roots of $O(1)$ endpoints, which can be ordered so that each radicand is a polynomial in the preceding ones, as Lemma~\ref{lem:compare} requires. Every endpoint is computed from the arc coefficients and the breakpoints of $\varphi$ alone, never from previously computed endpoints, so the depth of nesting never exceeds two. The endpoints computed during the sweep may have degree four over $\Q$.

We compute $H_1,\dots,H_K$ in turn by a sweep over the arcs from left to right; $H_K=\env\varphi$. The envelope $H_1$ is $\hat\varphi$ on the arc $1$. For the update, apply Lemma~\ref{lem:bridge} to $g:=\hat\varphi|_{[0,\xi_{j+1}]}$ with $\xi:=\xi_j$; the three envelopes of that lemma are then $H_j$, the restriction of $\hat\varphi$ to the arc $j+1$, which is convex, and $H_{j+1}$. Hence there are $\eta\le\xi_j\le\zeta$ with
\[
H_{j+1}=H_j \text{ on } [0,\eta],
\qquad
H_{j+1} \text{ affine on } [\eta,\zeta],
\qquad
H_{j+1}=\hat\varphi \text{ on } [\zeta,\xi_{j+1}].
\]
In particular the update deletes the contact pieces and bridges of $H_j$ lying to the right of $\eta$ and leaves the rest untouched. The arc $j+1$, which we call the incoming arc, is the graph of a strictly convex parabola or a segment. We scan the contact pieces of $H_j$ from right to left. Let $P$ be the rightmost contact piece, lying in the arc $i$, and let $w$ be its left endpoint. If $w=0$, then $P$ survives. Otherwise let $\ell_0$ be the line through $(w,H_j(w))$ with slope $H_j'(w^-)$, which is the line of the bridge ending at $w$ if there is one, and the tangent to the preceding contact piece at $w$ otherwise. By Lemma~\ref{lem:bridge}, $H_{j+1}=H_j$ on $[0,w]$ if and only if $\ell_0$ lies below the incoming arc, a test of $O(1)$ sign determinations: a line lies below the graph of a strictly convex parabola or of an affine function over a compact interval if and only if it does so at the two endpoints and, for a parabola, at the point of the interval where the two slopes agree, if there is one. If it does, $P$ survives. Otherwise $H_{j+1}$ differs from $H_j$ somewhere on $[0,w]$, so $\eta<w$, and $P$ is popped, together with the bridge ending at $w$ if there is one; the scan then continues with the preceding contact piece.

If $P$ survives, the new bridge is the common supporting line $\ell$ from below of the carrier of $P$ over $[w,\xi_i]$ and of the incoming arc, both of which are the graph, over a compact interval, of a strictly convex parabola with rational coefficients or of an affine function. Indeed, the line of $H_{j+1}$ on $[\eta,\zeta]$ lies below $H_j\le\hat\varphi$ and below the incoming arc, and touches both; and a line with these properties is unique, since between its two contact points it coincides with the convex envelope of the union of the two graphs, unless both contacts are at a common endpoint $\xi_i=\xi_j$, in which case $\eta=\zeta$ and no bridge is needed. The line $\ell$ touches each of them at an interior point, where it is tangent, or at an endpoint, and it is found among $O(1)$ candidates by $O(1)$ sign determinations: the common tangent of the two parabolas; the tangent from an endpoint $(\xi_j,\varphi(\xi_j))$ or $(\xi_{j+1},\varphi(\xi_{j+1}))$ of the incoming arc to the parabola carrying $P$; the tangent from $(\xi_i,\varphi(\xi_i))$, or from $(0,\varphi(0))$ if $w=0$, to the incoming parabola; and the line through one of these points on each side. Here a contact with the carrier of $P$ at a point of $(0,\xi_i)$ is a tangency, even at $w$, because such a point is a contact point of $H_j$, where $H_j$ is differentiable by Lemma~\ref{lem:tangent}. If $\ell$ touches the carrier of $P$ or the incoming arc on a nondegenerate interval, which happens only if that carrier or arc is affine, we take $\eta$ as large and $\zeta$ as small as possible, that is, $\eta=\xi_i$ or $\zeta=\xi_j$ in the respective case, so that the contact pieces stored are maximal. In every case $\ell$ touches the carrier of $P$ at a point $\eta\in[w,\xi_i]$ and the incoming arc at a point $\zeta$, both endpoints of the three kinds listed above; moreover $\eta\in P$, since $\ell$ lies below $H_j$ and $\ell(\eta)=\hat\varphi(\eta)$, so that $\eta$ is a contact point of $H_j$, and, by Lemma~\ref{lem:onecontact}, the contact set of $\hat\varphi|_{[0,\xi_j]}$ meets the arc $i$ only in $P$. Then $P$ is replaced by $[w,\eta]$, the incoming arc contributes the contact piece $[\zeta,\xi_{j+1}]$, and $\ell$ is the bridge between them, unless $\eta=\zeta$, in which case no bridge is stored; this completes the update.

It remains to count the operations. Each contact piece is created once and popped at most once during the whole sweep, and each survival test, pop, and bridge computation costs $O(1)$ arithmetic operations and $O(1)$ comparisons on rational numbers, so the sweep costs $O(K)$ of each.
\end{proof}

The second step reads off the value function from the envelope list.

\begin{lemma}
\label{lem:compute}
Let $\varphi\in\mathcal Q$ have $K$ arcs with rational coefficients and breakpoints given by representations, and let $s\in\mathbb Q\setminus\{0\}$. Then the arcs of $\mu(t)=\min_{x\in[0,1]}[\varphi(x)+stx]$ can be computed exactly with $O(K)$ arithmetic operations and comparisons on rational numbers. The arcs produced have rational coefficients, and each breakpoint produced is given by its defining quadratic and a representation. Moreover the arcs of $\mu$ are produced in the order of their breakpoints, increasing if $s<0$ and decreasing if $s>0$.
\end{lemma}

\begin{proof}
By Lemma~\ref{lem:envelope}, the envelope list of $\varphi$, which describes $H:=\env\varphi$, can be computed with $O(K)$ operations; as there, every comparison below costs $O(1)$ operations by Lemma~\ref{lem:compare}. We read off $\mu$ from $H$ by Lemma~\ref{lem:conjugate}, applied to the function $\hat\varphi$ of Lemma~\ref{lem:chord}, which has the same convex envelope and the same values at $0$ and $1$ as $\varphi$, hence the same $\mu$: each nondegenerate contact piece contained in a strictly convex arc $j$, with coefficients $(a_j,b_j,d_j)$, nondegeneracy being one comparison, contributes the arc $\mu(t)=-\frac{(st+b_j)^2}{4a_j}+d_j$; each bridge, and each contact piece in the list that is degenerate or whose carrier is a segment, contributes at most a breakpoint; and the two unbounded arcs are $\mu(t)=\varphi(0)$ and $\mu(t)=\varphi(1)+st$. This yields the arcs of $\mu$, in order and with rational coefficients, with $O(K)$ further operations. Consecutive arcs produced carry distinct polynomials, so no two need to be merged: two consecutive quadratic arcs come from contact pieces either in consecutive arcs of $\varphi$, which carry distinct polynomials, or separated by an interval on which $H$ is affine, whose line is tangent to both parabolas at two distinct points, so that the parabolas differ; a quadratic arc has $t^2$-coefficient $-s^2/(4a_j)\ne0$, while the two unbounded arcs are affine and distinct as $s\ne0$. A breakpoint $\tau$ of $\mu$ separates two consecutive arcs with polynomials $\pi_1\ne\pi_2$, and by continuity of $\mu$ it is a root of $\pi_1-\pi_2$, which we call its defining quadratic, as in Lemma~\ref{lem:rational}. Moreover $\tau=-\sigma/s$, where $\sigma$ is the slope of $H$ at the corresponding endpoint in the envelope list, or on the corresponding bridge, so deciding which root of $\pi_1-\pi_2$ equals $\tau$ is one comparison, after which a representation of $\tau$ is obtained with $O(1)$ operations; the nested radicals are then discarded. Lastly, the contact pieces in the envelope list are ordered by increasing $x$, hence in increasing order of the slope $\sigma=H'$, and the arc of $\mu$ that a contact piece contributes is indexed by $t=-\sigma/s$; the arcs of $\mu$ are therefore produced in the order of their breakpoints, increasing if $s<0$ and decreasing if $s>0$.
\end{proof}

It remains to recover an optimal solution of Problem~\eqref{eq:QP}.

\begin{lemma}
\label{lem:backtrack}
Problem~\eqref{eq:QP} admits an optimal solution $x^\star\in\mathbb Q^n$, with $\langle x^\star\rangle$ polynomial in $\L$, computable by backtracking from the functions $\varphi_v$ of \eqref{eq:message} with $O\big(\sum_vK_v\big)$ additional operations.
\end{lemma}

\begin{proof}
Consider first the root problem $\min_{x\in[0,1]}\varphi_r(x)$. By Lemma~\ref{lem:candidates}, applied to $\varphi_r|_{[0,1]}\in\mathcal Q$, the minimum is attained at one of the at most $K_r+2$ points listed there, which we call candidates. Each candidate is $0$, $1$, or is obtained from the coefficients of one arc by $O(1)$ operations, and whether it belongs to its arc is decided by two comparisons against the endpoints of the arc, which are given by representations. Evaluating $\varphi_r$ at the admissible candidates and comparing the values exactly selects $x_r^\star$.

Now descend. Given the value already chosen for the parent, $t^\star:=x^\star_{p(v)}$, we choose $x^\star_v$ as a minimizer of $\varphi_v(x)+s_vt^\star x$ over $[0,1]$. Since adding an affine function changes no one-sided derivative jump, $x\mapsto\varphi_v(x)+s_vt^\star x$ lies in $\mathcal Q$, with the same arcs as $\varphi_v|_{[0,1]}$ and with $b$ replaced by $b+s_vt^\star$ on each arc, so Lemma~\ref{lem:candidates} applies to it, and $x^\star_v$ is selected as at the root. The total operation count is $O\big(\sum_vK_v\big)$.

The resulting $x^\star$ is optimal. Indeed, by induction on the height of $v\ne r$, the point $x^\star_{T_v}$ minimizes $F_v(\cdot\,;x^\star_{p(v)})$ over $[0,1]^{T_v}$: by \eqref{eq:Fsplit}, the inductive hypothesis and Lemma~\ref{lem:dp}, $F_v(x^\star_{T_v};x^\star_{p(v)})=\varphi_v(x^\star_v)+s_vx^\star_{p(v)}x^\star_v$, which equals $\mu_v(x^\star_{p(v)})$ by the choice of $x^\star_v$ and \eqref{eq:message}. At the root, the same computation with $s_r=0$ gives $f(x^\star)=F_r(x^\star;0)=\varphi_r(x^\star_r)=\min_{x\in[0,1]}\varphi_r(x)$, which is the optimal value by Lemma~\ref{lem:dp}.

All candidates are rational. To bound their lengths, write $h(\alpha):=\max\{\lceil\log_2(|\nu|+1)\rceil,\lceil\log_2(\delta+1)\rceil\}$ for a rational $\alpha=\nu/\delta$ in lowest terms, so that $\langle\alpha\rangle\le 2h(\alpha)+2$. By Lemma~\ref{lem:bits}, the arc coefficients have $h$ bounded by some $\kappa$ polynomial in $\L$, and we may take $\kappa$ so large that $h(s_v)\le\kappa$ for all $v$. At the root, every candidate therefore has $h=O(\kappa)$. At a child, writing $t^\star=\nu/\delta$, the candidate $-(b+s_vt^\star)/(2a)$ is a quotient of integers whose absolute values are at most $2^{O(\kappa)}\max\{|\nu|,\delta\}$ and $2^{O(\kappa)}\delta$, so its $h$ exceeds $h(t^\star)$ by $O(\kappa)$. Hence $h(x^\star_v)=O(n\kappa)$ for every $v$, and $\langle x^\star\rangle$ is polynomial in $\L$, and so are the values compared when selecting each $x^\star_v$, which are obtained from the candidates, $t^\star$ and the arc coefficients by $O(1)$ operations.
\end{proof}

We are now ready to prove Theorem~\ref{thm:main}.

\begin{proof}[Proof of Theorem~\ref{thm:main}]
We first bound the number of operations. Process vertices leaf-to-root, and let $d_v:=|\ch(v)|$ denote the number of children of $v$. Let $m_v$ denote the total number of breakpoints of the $\mu_j$, $j\in\ch(v)$; by Proposition~\ref{prop:master}, $m_v\le\sum_{j\in\ch(v)}2|T_j|=2(|T_v|-1)$. By Lemma~\ref{lem:compute} the arcs of each child value function $\mu_j$ are produced in the order of their breakpoints, up to a global reversal according to the sign of $s_j$, so that the $d_v$ breakpoint lists are available in increasing order, after reversing some of them, at a cost of $O(m_v)$. Forming $\varphi_v$ amounts to merging them, discarding the breakpoints outside $(0,1)$; each comparison of two breakpoints costs $O(1)$ operations by Lemma~\ref{lem:compare}. We merge by a single scan that repeatedly selects the smallest among the $d_v$ current heads, at a cost of $O(d_vm_v)$ comparisons, while maintaining the sum of the arc coefficients incrementally: crossing a breakpoint of a child $\mu_j$ updates the running sum by subtracting the coefficients of the arc of $\mu_j$ that ends there and adding those of the arc that begins, at a cost of $O(1)$ operations, the initial sum costing $O(d_v)$. Several children may have a breakpoint at the same point, in which case all their updates are performed there; and consecutive intervals of the merged list carrying the same polynomial are coalesced, so that the arcs of $\varphi_v$ are maximal. The merge thus costs $O(d_vm_v+d_v)$ operations and comparisons, including the $O(m_v)$ for reversing. For $v\ne r$, computing $\mu_v$ then costs $O(K_v)$ operations by Lemma~\ref{lem:compute} applied to $(\varphi_v|_{[0,1]},s_v)$; its hypotheses hold because $\varphi_v|_{[0,1]}\in\mathcal Q$ (Proposition~\ref{prop:master}), its arcs have rational coefficients and its breakpoints are breakpoints of the $\mu_j$ (Lemma~\ref{lem:rational}), whose representations are provided by Lemma~\ref{lem:compute}, and $s_v=2q_{p(v)v}\in\Q\setminus\{0\}$. By Proposition~\ref{prop:master}, $K_v\le 2|T_v|-1$, and
\[
\sum_{v\in V}|T_v|=\sum_{u\in V}\big(\mathrm{depth}(u)+1\big)\le n^2 ,
\]
because $\mathrm{depth}(u)\le n-1$ for every $u$; hence $\sum_{v\in V}K_v=O(n^2)$. The merging is equally cheap: since $m_v\le 2(|T_v|-1)\le 2n$ and $\sum_{v\in V}d_v=n-1$, we have $\sum_{v\in V}(d_vm_v+d_v)\le(2n+1)(n-1)=O(n^2)$. Adding the $O(\sum_vK_v)$ operations of the backtracking of Lemma~\ref{lem:backtrack}, the total is $O(n^2)$ arithmetic operations and comparisons. Lemma~\ref{lem:backtrack} then returns an optimal solution.

We now bound the lengths of the numbers involved. The following rationals have length polynomial in $\L$: the entries of $Q$ and $c$; the arc coefficients of the $\varphi_v$ and $\mu_v$, the coefficients of their defining quadratics and the rationals of the representations of their breakpoints, by Lemma~\ref{lem:bits}; the partial sums of at most $n$ arc coefficients formed during a merge, as sums of at most $n$ such rationals; and the coordinates of $x^\star$ and the values compared when selecting them, by Lemma~\ref{lem:backtrack}. Every other number computed by the algorithm is obtained from these by $O(1)$ arithmetic operations: this holds for the rationals representing the auxiliary nested radicals of Lemma~\ref{lem:envelope}, since every endpoint there is computed from arc coefficients and breakpoints alone, and for the rationals computed within a sign determination of Lemma~\ref{lem:compare}. Hence every number computed has length polynomial in $\L$.

Finally, the operation count depends on $n$ alone and every intermediate number has length polynomial in the input length, so the algorithm is strongly polynomial.
\end{proof}

\subsection{Relation to the literature}
\label{sec:relation}

Kolmogorov, Pock and Rolinek~\cite{KolPocRol16} consider problems over a tree of the form
\[
\min\Big\{\sum_{v\in V}g_v(x_v)+\sum_{v\ne r}h_v\big(x_v-x_{p(v)}\big)\ :\ x\in\R^V\Big\},
\]
and solve them by dynamic programming when the terms $h_v$ are (truncated) total variation terms; Kuric, Ahmetspahic and Pock~\cite{KurAhmPoc24} extend this type of algorithm to piecewise quadratic $g_v$ and $h_v$, motivated by total generalized variation models. Up to the fact that in~\cite{KurAhmPoc24} these terms are finite on $\R$, Problem~\eqref{eq:QP} is of this form when its interaction graph is a tree: by the identity $s_vx_{p(v)}x_v=\tfrac12s_v(x_{p(v)}^2+x_v^2)-\tfrac12s_v(x_v-x_{p(v)})^2$, one can take $h_v(d)=-\tfrac12s_vd^2$, and $g_v(x)=\big(q_{vv}+\sum_{u:\{u,v\}\in E}q_{uv}\big)x^2+c_vx$ on $[0,1]$ and $g_v(x)=+\infty$ elsewhere. The message sent from $v$ to $p(v)$ by the dynamic program of~\cite{KurAhmPoc24} is then $\mu_v(t)-q_{p(v)v}t^2$. For nonconvex terms, the algorithm of~\cite{KurAhmPoc24} has exponential worst-case time and memory complexity, whereas, for the terms arising from Problem~\eqref{eq:QP}, the value functions have linearly many arcs (Proposition~\ref{prop:master}), and Theorem~\ref{thm:main} gives a strongly polynomial algorithm performing $O(n^2)$ operations.

A second, more closely related, line of work concerns \emph{convex quadratic optimization with indicator variables,} that is, the problem 
\[
\min\Big\{\tfrac12x^\top Qx+c^\top x+\lambda^\top z \ :\ x_i(1-z_i)=0 \text{ for } i\in[n],\ x\in\R^n,\ z\in\{0,1\}^n\Big\},
\]
where $Q\succ0$ and $\lambda>0$.
When the interaction graph of $Q$ is a tree, Bhathena, Fattahi, G\'omez and K\"u\c{c}\"ukyavuz~\cite{BhaFatGomKuc26} give an $O(n^2)$ algorithm based on a parametric analysis of the value functions of a dynamic program over the tree.
For positive definite $Q$ whose interaction graph has bounded treewidth, the same authors give a parametric dynamic program over a tree decomposition~\cite{BhaFatGomKuc26b}, whose complexity depends on structural and numerical parameters, including treewidth, volume growth, conditioning and a margin parameter, and is linear in $n$ when the relevant parameters are bounded appropriately; Problem~\eqref{eq:QP}, instead, is strongly $\NP$-hard already when the interaction graph has treewidth two (Theorem~\ref{thm:strong-NP-hardness-tw2}).
The technique of~\cite{BhaFatGomKuc26} is closely related to ours.

\paragraph{Similarities.} Both algorithms process the tree from the leaves to the root. Their \emph{parametric cost} at $v$, the optimal value of the subproblem on the subtree rooted at $v$ as a function of $x_v$, plays the role of our $\varphi_v$, and their recursion is the analogue of \eqref{eq:message}: for $v\ne r$, minimizing the parametric cost at $v$ plus $q_{p(v)v}\,x_{p(v)}x_v$ over $x_v$ yields a function of $x_{p(v)}$, expressed in~\cite{BhaFatGomKuc26} through a Fenchel conjugate, which corresponds to our value function $\mu_v$. Both analyses isolate a class of piecewise quadratic functions preserved by the recursion, their \emph{consistent} functions, up to an indicator term, and our class $\mathcal Q$, and share three ingredients: conjugation adds at most two arcs (\cite[Proposition~1]{BhaFatGomKuc26} and Lemma~\ref{lem:conjugate}); the conjugate is computed by a linear-time sweep (\cite[Proposition~2, Theorem~3]{BhaFatGomKuc26} and Lemmas~\ref{lem:envelope} and~\ref{lem:compute}); and every arc has a closed form (\cite[Lemma~1]{BhaFatGomKuc26} and Lemma~\ref{lem:activeset}). In both, the number of arcs adds at vertices with several children, which yields $O(n^2)$ operations.

\paragraph{Differences.} The two works differ in the problem solved, in the role of convexity, and in the model of computation; the last two are where the main difficulties of our analysis lie.
\medskip
\begin{enumerate}[leftmargin=0.8cm]
\item [$(i)$] \emph{Problem structure.} Neither problem is a special case of the other. The problem of~\cite{BhaFatGomKuc26} is mixed-integer, with unbounded continuous variables whose support is chosen by the binary variables and penalized by $\lambda^\top z$; Problem~\eqref{eq:QP} is continuous, with variables confined to a box. Accordingly, the coordinates that are not free in Lemma~\ref{lem:activeset} sit at either endpoint of their range, whereas in~\cite{BhaFatGomKuc26} they are all equal to zero.

\item [$(ii)$] \emph{Nonconvexity.} In~\cite{BhaFatGomKuc26}, $Q\succ0$, so every arc of every parametric cost, up to the indicator term, is a strongly convex quadratic~\cite[Lemma~1]{BhaFatGomKuc26}, and this is why consistency is preserved along the recursion. We make no assumption on $Q$: an arc of $\varphi_v$ may be strictly convex, affine or strictly concave, and consistency fails. What survives is the weaker property that every kink is concave (Proposition~\ref{prop:master}), which defines $\mathcal Q$ and holds for the restriction to $[0,1]$ of every consistent function. Contact must then be excluded from the interior of strictly concave arcs (Lemma~\ref{lem:noconcave}), and the sweep must replace every strictly concave arc by its chord (Lemma~\ref{lem:chord}); both steps are vacuous in the convex setting. The count of at most one arc of $\mu_v$ per arc of $\varphi_v$ (Lemmas~\ref{lem:noconcave} and~\ref{lem:onecontact}) is the analogue of the monotonicity of their indexing function~\cite[Lemma~5]{BhaFatGomKuc26}, and holds here for arcs of any curvature. Finally, the positive definiteness of the free block becomes a conclusion instead of a hypothesis: in Lemma~\ref{lem:activeset}, $Q_S\succ0$ follows from Lemma~\ref{lem:pdface}, which makes the length bounds of Lemma~\ref{lem:bits} available.

\item [$(iii)$] \emph{Model of computation.} The bound $O(n^2)$ of~\cite{BhaFatGomKuc26} counts arithmetic operations on real numbers and does not estimate the lengths of the numbers; already the normalization $q_{ii}=1$ made there, by $x_i\leftarrow x_i/\sqrt{q_{ii}}$, takes the data out of $\Q$. We work in the Turing model. In both settings the breakpoints are irrational already for rational data, for instance slopes of common tangents to two parabolas; we compute and compare them exactly, with $O(1)$ operations on rational numbers and no root extraction (Lemma~\ref{lem:compare}), whereas~\cite[Section~4]{BhaFatGomKuc26} reports that in floating-point arithmetic their computation becomes unstable on large trees and gives a correction step. That the numbers produced along the recursion have polynomial length is not evident from the recursion itself; we obtain it from the closed form of Lemma~\ref{lem:activeset}, which exhibits every arc as the value of an unconstrained quadratic minimization with a positive definite matrix (Lemma~\ref{lem:bits}). 

We believe that our techniques can be applied to prove that the algorithm of~\cite{BhaFatGomKuc26} is in fact strongly polynomial, once we remove the normalization $q_{ii}=1$. Namely, by~\cite[Lemma~1]{BhaFatGomKuc26}, every arc of their parametric costs is also the value of an unconstrained quadratic minimization with a positive definite matrix, so the argument of Lemma~\ref{lem:bits} bounds the lengths of its coefficients, and their breakpoints can be computed and compared exactly as in Lemma~\ref{lem:compare}.

\end{enumerate}

\subsection{Higher degree box-constrained polynomial optimization}
\label{sec:quartic}

Theorem~\ref{thm:main} shows that a box-constrained quadratic program can be solved in strongly polynomial-time when the interaction graph is a forest. It is therefore natural to ask whether a similar tractability result continues to hold for polynomials of higher degree under the same sparsity assumption. The next result shows that this is not the case: already for quartic objectives, the problem becomes strongly $\NP$-hard even when the interaction graph is a path. Thus, the quadratic structure of Problem~\eqref{eq:QP}, and not only the acyclicity of its interaction graph, is essential to our tractability result.

Let
$f(x)=\sum_{\alpha\in\mathbb Z_{\ge0}^n}c_\alpha x^\alpha$
be a polynomial, where \(x^\alpha:=\prod_{i=1}^n x_i^{\alpha_i}\). We define the \emph{interaction graph} of $f$ as the graph $G=(V,E)$ with vertex set $V=[n]$, containing one vertex for each variable $x_i$, and with
$\{i,j\} \in E$ for some $i\neq j \in [n]$ if and only if there exists $\alpha$ with $c_\alpha\ne0$ such that $\alpha_i>0$ and $\alpha_j>0$.
Namely, two distinct vertices are adjacent if and only if the corresponding variables appear together in a monomial with nonzero coefficient.

\begin{theorem}\label{thm:quartic-path}
Minimizing a polynomial of degree four with integer coefficients over the
unit hypercube is strongly $\NP$-hard even when its interaction graph is a
path. 
\end{theorem}

\begin{proof}
Consider the decision version of the problem: given a degree four polynomial $f(x)$ with integer coefficients and $\gamma \in \Q$, decide whether 
$$
\min\{f(x): x \in [0,1]^n\} \leq \gamma.
$$
To prove the statement, it suffices to show that this decision problem is strongly $\NP$-hard.
We give a polynomial-time reduction from the strongly $\NP$-complete \textsc{Product Partition} problem~\cite{NgBarCheKov10}.
Given positive integers $a_1,\dots,a_n$, the \textsc{Product Partition}
problem asks whether there exists a subset $S\subseteq[n]$ such that
\begin{equation}\label{eq:product-partition}
    \prod_{i\in S} a_i
    =
    \prod_{i\in[n]\setminus S} a_i.
\end{equation}
Define $P:=\prod_{i=1}^n a_i$, introduce the variables
$x_1,\dots,x_{2n}\in[0,1]$, 
and let $x_0:=1$.
Consider the polynomial
\begin{align}
f(x):=
&\sum_{i=1}^{n}
    \left(a_i x_i-x_{i-1}\right)^2
\label{eq:quartic-objective-a}\\
&+
\sum_{i=1}^{n}
    (x_{n+i}-x_{n+i-1})^2
    (x_{n+i}-a_i^2x_{n+i-1})^2
\label{eq:quartic-objective-b}\\
&+
(x_{2n}-1)^2.
\label{eq:quartic-objective-c}
\end{align}
Clearly, $f(x)$ is a polynomial of degree four with integer coefficients.
We first prove that the minimum of $f(x)$ over $[0,1]^{2n}$ is zero
if and only if the given \textsc{Product Partition} instance is feasible.
First, suppose that $x\in[0,1]^{2n}$ satisfies $f(x)=0$.
Since every term in~\eqref{eq:quartic-objective-a}--%
\eqref{eq:quartic-objective-c} is nonnegative, every squared term must
vanish. From~\eqref{eq:quartic-objective-a} we obtain
$a_i x_i=x_{i-1}$ for all $i\in[n]$.
Using $x_0=1$ recursively gives
\begin{equation}\label{eq:first-chain}
    x_i=\frac{1}{\prod_{j=1}^{i}a_j},
    \quad \forall i\in[n].
\end{equation}
and in particular
\begin{equation}\label{eq:yn}
    x_n=\frac1P.
\end{equation}
Next, from~\eqref{eq:quartic-objective-b}, for every $i\in[n]$ we obtain
$$
    (x_{n+i}-x_{n+i-1})
    (x_{n+i}-a_i^2x_{n+i-1})=0.
$$
Hence, for each $i\in[n]$, either
$x_{n+i}=x_{n+i-1}$ or $x_{n+i}=a_i^2x_{n+i-1}$.
Let $S\subseteq[n]$ be the set of indices for which $x_{n+i}=a_i^2x_{n+i-1}$. 
Using~\eqref{eq:yn}, we obtain
\begin{equation}\label{eq:second-chain}
    x_{2n}
    =\frac1P\prod_{i\in S}a_i^2.
\end{equation}
Moreover, from~\eqref{eq:quartic-objective-c} we obtain $x_{2n}=1$, implying that
\[
    \prod_{i\in S}a_i^2=P
    =\prod_{i\in [n]}a_i=
    \Big(\prod_{i\in S}a_i\Big)
    \Big(\prod_{i\in[n]\setminus S}a_i\Big).
\]
Since all $a_i$s are positive, we deduce that
$\prod_{i\in S}a_i
    =  \prod_{i\in[n]\setminus S}a_i$.
Thus $S$ solves the \textsc{Product Partition} instance.

Conversely, suppose that $S\subseteq[n]$ satisfies
\eqref{eq:product-partition}. Define
\begin{equation}\label{eq:y-first}
    x_i:=
    \frac{1}{\prod_{j=1}^{i}a_j},
    \quad \forall i\in[n].
\end{equation}
Since $a_j \geq 1$ for all $j \in [n]$, we have $x_i \in (0,1]$ for all $i \in [n]$.
For each $i\in[n]$, define
\begin{equation}\label{eq:y-second}
    x_{n+i}
    :=
    \frac1P
    \prod_{\substack{j\in S\\j\le i}}a_j^2.
\end{equation}
Since $a_j \geq 1$, we have $\prod_{j\in S, \; j \leq i}a_j^2 \leq \prod_{j\in S}a_j^2$. From~\eqref{eq:product-partition} it follows that $\prod_{j\in S}a_j^2=P$.  
Therefore, $0<x_{n+i}\le1$, for all $i\in[n]$. We then have $x\in[0,1]^{2n}$.
From~\eqref{eq:y-first} it follows that
$$
    a_i x_i-x_{i-1}=0
    \quad\forall i\in[n],
$$
implying that every term in~\eqref{eq:quartic-objective-a} vanishes.
Furthermore, from~\eqref{eq:y-second} it follows that
\[
    x_{n+i}
    =
    \begin{cases}
      a_i^2x_{n+i-1}, & i\in S,\\
      x_{n+i-1}, & i\notin S,
    \end{cases} \quad \forall i \in [n].
\]
Therefore, every term in~\eqref{eq:quartic-objective-b} also vanishes.
Finally, by construction we have $x_{2n}=1$, so~\eqref{eq:quartic-objective-c} vanishes as well.
Therefore $f(x)=0$.

We next examine the interaction graph of $f(x)$. Expanding the objective function gives
\begin{align*}
f(x)
=&
\sum_{i=1}^{n}
\big(
a_i^2x_i^2
-2a_i x_{i-1}x_i
+x_{i-1}^2
\big)\\
+&\sum_{i=1}^{n}
\big(
x_{n+i}^4
-2(a_i^2+1)x_{n+i-1}x_{n+i}^3
+(a_i^4+4a_i^2+1)x_{n+i-1}^2x_{n+i}^2
-2a_i^2(a_i^2+1)x_{n+i-1}^3x_{n+i}
+a_i^4x_{n+i-1}^4
\big)\\
&+x_{2n}^2-2x_{2n}+1,
\end{align*}
where $x_0:=1$.
It is then simple to see that the interaction graph of $f(x)$ is the path $x_1-x_2-\cdots-x_{2n}$.

Finally, we verify strong $\NP$-hardness. The polynomial $f(x)$ has
$2n$ variables and $O(n)$ monomials after expansion. Let
$a_{\max}:=\max_{i\in[n]}a_i$.
It follows from the above expansion that every coefficient of $f(x)$
has absolute value at most $6a_{\max}^4$. Since
\textsc{Product Partition} is strongly $\NP$-complete
\cite{NgBarCheKov10}, we may restrict attention to instances for which
$a_{\max}$ is polynomially bounded in $n$. Hence all numerical
coefficients of the constructed instance are polynomially bounded in
its number of variables. The box bounds are $0$ and $1$, and we set
the decision threshold to $\gamma:=0$. The construction is clearly
polynomial in the encoding length of the \textsc{Product Partition}
instance. Therefore, the decision problem is strongly $\NP$-hard.
\end{proof}

\section{Treewidth two: Strong $\NP$-hardness}
\label{sec: NP-hard}

In this section, we show that Problem~\eqref{eq:QP} is strongly $\NP$-hard even when the interaction graph has treewidth two. This result in turn implies that Theorem~\ref{thm:main} is sharp with respect to the treewidth.
Let us first give a formal definition of treewidth.
Given a graph $G = (V, E)$, a \emph{tree decomposition} of $G$ is a pair $(\X, T)$, where $\X = \{X_1, \dots, X_m\}$ is a family of subsets of $V$, called \emph{bags}, and $T$ is a tree with $m$ vertices, the $i$-th vertex being associated with the bag $X_i$, such that:

\medskip

\begin{enumerate}
    \item $V = \bigcup_{i \in [m]}{X_i}$.
    \item For every edge $\{u,v\} \in E$, there is a bag $X_i$, $i \in [m]$, with $u,v \in X_i$.
    \item For each vertex $v \in V$, the set of all bags containing $v$ induces a connected subtree of $T$.
\end{enumerate}
\medskip
The \emph{width} of a tree decomposition $(\X,T)$ is the size of its largest bag $X_i$ minus one. The \emph{treewidth} $\tw(G)$ of a graph $G$ is the minimum width among all possible tree decompositions of $G$.  

\begin{figure}[t]
    \centering
    \scalebox{0.85}{\begin{tikzpicture}[
        vplus/.style={circle,draw,fill=black,minimum size=5.5pt,inner sep=0pt},
        nbr/.style={circle,draw,fill=black!30,minimum size=5.5pt,inner sep=0pt},
        rest/.style={circle,draw,fill=white,minimum size=5.5pt,inner sep=0pt},
        blk/.style={draw,rounded corners=2pt,fill=white,minimum width=1.4cm,minimum height=0.75cm,inner sep=2pt},
        every label/.style={inner sep=1.5pt,font=\small}
    ]
        \node[vplus,label=below:{$s_1$}]     (s1) at (0,0) {};
        \node[vplus,label=below:{$s_2$}]     (s2) at (2.2,0) {};
        \node[vplus,label=below:{$s_{n-1}$}] (sm) at (5.0,0) {};
        \node[vplus,label=below:{$s_n$}]     (sn) at (7.2,0) {};
        \draw (s1) -- (s2);
        \draw (s2) -- (2.95,0);  \node at (3.6,0) {$\cdots$};  \draw (4.25,0) -- (sm);
        \draw (sm) -- (sn);
        \node[vplus,label=left:{$z_{1\ell}$}] (z1) at (-0.70,1.20) {};
        \node[vplus,label=left:{$z_{2\ell}$}] (z2) at (1.10,1.20) {};
        \node[vplus,label=left:{$z_{n\ell}$}] (zn) at (6.10,1.20) {};
        \draw (z1) -- (s1);
        \draw (z2) -- (s1);  \draw (z2) -- (s2);
        \draw (zn) -- (sm);  \draw (zn) -- (sn);
        \node[blk] (B1) at (-0.70,2.55) {$B_1$};
        \node[blk] (B2) at (1.10,2.55) {$B_2$};
        \node[blk] (Bn) at (6.10,2.55) {$B_n$};
        \draw (z1) -- ($(B1.south)+(-0.30,0)$);  \draw (z1) -- ($(B1.south)+(0.30,0)$);
        \draw (z2) -- ($(B2.south)+(-0.30,0)$);  \draw (z2) -- ($(B2.south)+(0.30,0)$);
        \draw (zn) -- ($(Bn.south)+(-0.30,0)$);  \draw (zn) -- ($(Bn.south)+(0.30,0)$);
        \node at (3.6,1.9) {$\cdots$};
        \node[vplus,label=below:{$w_\ell$}]     (wl) at (8.3,0) {};
        \node[vplus,label=below:{$w_{\ell-1}$}] (wm) at (9.5,0) {};
        \node[vplus,label=below:{$w_1$}]        (w1) at (11.7,0) {};
        \draw (sn) -- (wl) -- (wm);
        \draw (wm) -- (10.25,0);  \node at (10.6,0) {$\cdots$};  \draw (10.95,0) -- (w1);
        \node[anchor=north] at (5.5,-0.95) {(a)};
    \end{tikzpicture}}

    \vspace{1.2em}

    \scalebox{0.85}{\begin{tikzpicture}[
        vplus/.style={circle,draw,fill=black,minimum size=5.5pt,inner sep=0pt},
        nbr/.style={circle,draw,fill=black!30,minimum size=5.5pt,inner sep=0pt},
        rest/.style={circle,draw,fill=white,minimum size=5.5pt,inner sep=0pt},
        every label/.style={inner sep=1.5pt,font=\small}
    ]
        \node[vplus,label=below:{$z_{i1}$}]          (z1) at (0,0) {};
        \node[vplus,label=below:{$z_{i2}$}]          (z2) at (1.4,0) {};
        \node[vplus,label=below:{$z_{i3}$}]          (z3) at (2.8,0) {};
        \node[vplus,label=below:{$z_{i,\ell-1}$}]    (zp) at (4.9,0) {};
        \node[vplus,label=above right:{$z_{i\ell}$}] (zl) at (6.3,0) {};
        \node[vplus,label=above:{$x_{i1}$}]          (x1) at (-0.7,1.3) {};
        \node[vplus,label=above:{$x_{i2}$}]          (x2) at (0.7,1.3) {};
        \node[vplus,label=above:{$x_{i3}$}]          (x3) at (2.1,1.3) {};
        \node[vplus,label=above:{$x_{i\ell}$}]       (xl) at (5.6,1.3) {};
        \draw (z1) -- (z2) -- (z3);
        \draw (z3) -- (3.45,0);  \node at (3.85,0) {$\cdots$};  \draw (4.25,0) -- (zp);
        \draw (zp) -- (zl);
        \draw (x1) -- (x2) -- (x3);
        \draw (x3) -- (2.75,1.3);  \node at (3.85,1.3) {$\cdots$};  \draw (4.95,1.3) -- (xl);
        \draw (x1) -- (z1);
        \draw (x2) -- (z1);  \draw (x2) -- (z2);
        \draw (x3) -- (z2);  \draw (x3) -- (z3);
        \draw (xl) -- (zp);  \draw (xl) -- (zl);
        \node[vplus,label=below:{$s_{i-1}$}] (sa) at (5.2,-1.2) {};
        \node[vplus,label=below:{$s_i$}]     (sb) at (7.4,-1.2) {};
        \draw[densely dashed] (zl) -- (sa);
        \draw[densely dashed] (zl) -- (sb);
        \draw[densely dashed] (sa) -- (sb);
        \node[anchor=north] at (3.5,-1.75) {(b)};
    \end{tikzpicture}}
    \caption{
    The interaction graph of the instance of Problem~\eqref{eq:QP}
    constructed in the proof of
    Theorem~\ref{thm:strong-NP-hardness-tw2}, drawn for $b_{ik}=1$ for
    all $i\in[n]$ and $k\in[\ell]$; if $b_{ik}=0$, the two edges of
    $x_{ik}$ incident to $z_{i,k-1}$ and $z_{ik}$ are absent.
    In~(a), the vertices $s_1,\dots,s_n$ form a path, the vertices
    $w_1,\dots,w_\ell$ form a path attached to $s_n$, and, for $i\geq2$,
    the vertex $z_{i\ell}$ forms a triangle with $s_{i-1}$ and $s_i$,
    while $z_{1\ell}$ is adjacent to $s_1$ alone. In either case
    $z_{i\ell}$ is the unique vertex through which the block $B_i$,
    depicted in~(b), is attached to the rest of the graph, and the two
    edges leaving $z_{i\ell}$ upwards lead to $z_{i,\ell-1}$ and
    $x_{i\ell}$. 
    }
    \label{fig:tw2-hardness}
\end{figure}
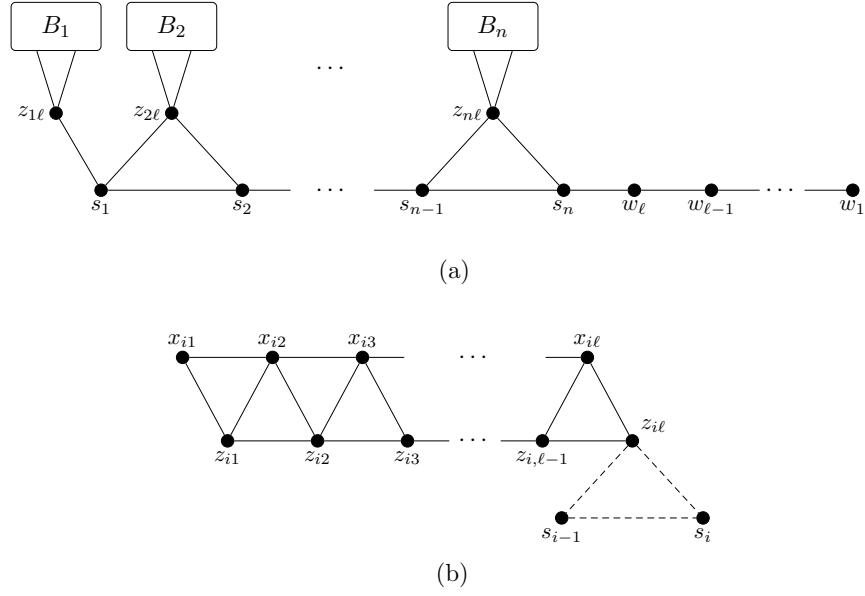

\medskip
We are now ready to present our $\NP$-hardness result.

\begin{theorem}\label{thm:strong-NP-hardness-tw2}
Problem~\eqref{eq:QP} is strongly $\NP$-hard even when the interaction graph has
treewidth two. Moreover, the result holds for instances satisfying
\[
    Q\in\mathbb{Z}^{n\times n},
    \quad
    c\in\mathbb{Z}^n,
    \quad
    \|Q\|_{\max}\leq 5,
    \quad
    \|c\|_{\infty}\leq 4,
\]
where $\|Q\|_{\max}:=\max_{p,q}|Q_{pq}|$ and $n$ is the number of variables of the instance. 
\end{theorem}
It is well-known that quadratic programming is in $\NP$~\cite{Vav90}. Therefore, Theorem~\ref{thm:strong-NP-hardness-tw2} implies that Problem~\eqref{eq:QP} is strongly $\NP$-complete, even if the interaction graph has treewidth two.

\begin{proof}
Consider the decision version of Problem~\eqref{eq:QP}: given
symmetric $Q\in\mathbb{Q}^{n\times n}$, $c\in\mathbb{Q}^n$, and
$\gamma\in\mathbb{Q}$, decide whether
$$
    \min\left\{x^\top Q x+c^\top x:x\in[0,1]^n\right\}
    \leq \gamma.
$$
To prove the statement, it suffices to show that this decision problem is strongly $\NP$-hard. 
We give a polynomial-time reduction from \textsc{Subset Sum}~\cite{GarJohBook}. Throughout this proof, $n$ denotes the number of items of the \textsc{Subset Sum} instance, and $N$ the number of variables of the constructed instance of Problem~\eqref{eq:QP}.
Given integers
$a_1,\dots,a_n \in \mathbb{Z}_{\ge 0}$ and a target $T \in \mathbb{Z}_{\ge 0}$, the \textsc{Subset Sum} asks
whether there exists $x \in \{0,1\}^n$ such that
\[
\sum_{i=1}^n a_i x_i = T.
\]
Define
\[
    U:=\sum_{i=1}^n a_i.
\]
The cases with $T=0$, $U=0$, and $T>U$ can be decided trivially. We therefore assume that $1\leq T\leq U$.
Define
\[
    \ell:=\left\lceil\log_2(U+1)\right\rceil
    \qquad\text{and}\qquad
    D:=2^\ell.
\]
In particular, we have $D>U$.
For each $i\in [n]$, write the binary expansion of $a_i$ as
\begin{equation}\label{aiexpand}
    a_i=\sum_{k=1}^{\ell}b_{ik}2^{k-1},
    \qquad b_{ik}\in\{0,1\}.
\end{equation}
Similarly, write the binary expansion of $T$ as
\begin{equation}\label{texpand}
    T=\sum_{k=1}^{\ell}t_k2^{k-1},
    \qquad t_k\in\{0,1\}.
\end{equation}
We construct an instance of Problem~\eqref{eq:QP} as follows.
Introduce the following variables:
\medskip
\begin{itemize}
    \item $s_i \in [0,1]$ for all $i \in [n]$,
    \item $w_k \in [0,1]$ for all $k \in [\ell]$,
    \item $x_{ik},z_{ik}\in[0,1]$ for all $i\in [n]$ and all $k\in[\ell]$.
\end{itemize}
\medskip
We further define 
$z_{i0}:=0$ for all $i \in [n]$, $s_0:=0$, and  $w_0:=0$.
Consider the following quadratic function:
\begin{align}
\Psi(x,z,s,w)
:={}&
\sum_{i=1}^n
\Big(
    x_{i1}(1-x_{i1})
    +\sum_{k=2}^{\ell}(x_{ik}-x_{i,k-1})^2
\Big)
\label{eq:strong-reduction-objective-a}\\
&+
\sum_{i=1}^n\sum_{k=1}^{\ell}
\left(
    2z_{ik}-z_{i,k-1}-b_{ik}x_{ik}
\right)^2
\label{eq:strong-reduction-objective-b}\\
&+
\sum_{i=1}^n
\left(
    s_i-s_{i-1}-z_{i\ell}
\right)^2
\label{eq:strong-reduction-objective-c}\\
&+
\sum_{k=1}^{\ell}
\left(
    2w_k-w_{k-1}-t_k
\right)^2
\label{eq:strong-reduction-objective-d}\\
&+
(s_n-w_\ell)^2.
\label{eq:strong-reduction-objective-e}
\end{align}
See Figure~\ref{fig:tw2-hardness} for the interaction graph of the resulting instance. Every term in~\eqref{eq:strong-reduction-objective-a}--%
\eqref{eq:strong-reduction-objective-e} is nonnegative over the unit box.
Consequently, $\Psi(x,z,s,w)\geq 0$
for every feasible point.
We first show that the minimum of $\Psi$ is zero if and only if the original
\textsc{Subset Sum} instance is feasible. Suppose that
$\Psi(x,z,s,w)=0$. Since all the terms in~\eqref{eq:strong-reduction-objective-a}--%
\eqref{eq:strong-reduction-objective-e} are nonnegative over the unit box, each of them vanishes.
The first term in~\eqref{eq:strong-reduction-objective-a} implies that
$x_{i1}\in\{0,1\}$ for all $i\in [n]$.
The remaining terms in~\eqref{eq:strong-reduction-objective-a} imply that
\[
    x_{ik}=x_{i,k-1}
    \quad
    \forall i\in [n], \;
    k\in\{2,\dots,\ell\}.
\]
Thus, for every $i$, there exists $\bar{x}_i\in\{0,1\}$ such that
\[
    x_{ik}=\bar{x}_i
    \quad
    \forall k\in [\ell].
\]
From~\eqref{eq:strong-reduction-objective-b}, we get
$z_{ik}=\frac{1}{2}(z_{i,k-1}+b_{ik}\bar{x}_i)$ for all $i \in [n]$ and all $ k \in [\ell]$,
which together with $z_{i0} = 0$ implies that
$$
    z_{ik}
    =
    \frac{\bar{x}_i}{2^k}
    \sum_{r=1}^k b_{ir}2^{r-1}
    \quad
    \forall i \in [n], \;
    k\in [\ell].
$$
In particular, from~\eqref{aiexpand} it follows that
\begin{equation}\label{eq:z-final-value}
    z_{i\ell}
    =
    \frac{a_i\bar{x}_i}{D}.
\end{equation}
Similarly, from
\eqref{eq:strong-reduction-objective-d} we get
$w_k=\frac{1}{2}(w_{k-1}+t_k)$ for all $k \in [\ell]$, which together with $w_0 = 0$ implies:
$$
    w_k
    =
    \frac{1}{2^k}
    \sum_{r=1}^kt_r2^{r-1}
    \quad \forall k\in [\ell].
$$
In particular, from~\eqref{texpand} it follows that
\begin{equation}\label{eq:w-final-value}
    w_\ell=\frac{T}{D}.
\end{equation}
From~\eqref{eq:strong-reduction-objective-c} we get
$s_i=s_{i-1}+z_{i\ell}$ for all  $i \in [n]$.
Using $s_0=0$ and~\eqref{eq:z-final-value}, we obtain
\begin{equation}\label{eq:s-final-value}
    s_n
    =
    \sum_{i=1}^n z_{i\ell}
    =
    \frac{1}{D}\sum_{i=1}^n a_i\bar{x}_i.
\end{equation}
Finally, from~\eqref{eq:strong-reduction-objective-e} we deduce that 
\begin{equation}\label{sn}
s_n=w_\ell.
\end{equation}
Combining
\eqref{eq:w-final-value},~\eqref{eq:s-final-value}, and~\eqref{sn} gives
$\sum_{i=1}^n a_i\bar{x}_i=T$.
Hence $\bar{x}\in\{0,1\}^n$ is a solution of the 
\textsc{Subset Sum} instance.

Conversely, suppose that $\bar{x}\in\{0,1\}^n$ satisfies
$\sum_{i=1}^n a_i\bar{x}_i=T$.
Set
$x_{ik}:=\bar{x}_i$ for all $i \in [n]$ and all $k\in [\ell]$,
and define
\begin{align*}
   & z_{ik}
    :=
    \frac{\bar{x}_i}{2^k}
    \sum_{r=1}^k b_{ir}2^{r-1} \quad \forall i \in [n], k \in [\ell]
\\
   & w_k
    :=
    \frac{1}{2^k}
    \sum_{r=1}^kt_r2^{r-1} \quad \forall k \in [\ell]
\\
   & s_i
    :=
    \frac{1}{D}\sum_{j=1}^i a_j\bar{x}_j \quad \forall i \in [n].
\end{align*}
It can be checked that substituting these values in
\eqref{eq:strong-reduction-objective-a}--%
\eqref{eq:strong-reduction-objective-e}, we get $\Psi=0$. They also belong to the unit box. Indeed, for every $i$ and $k$,
\[
    0\leq z_{ik}
    \leq
    \frac{1}{2^k}\sum_{r=1}^k2^{r-1}
    =
    1-\frac{1}{2^k}
    <1,
\]
and the same argument applies to $w_k$. Moreover,
$0\leq s_i \leq\frac{U}{D}<1$.
We have established that the minimum value of $\Psi(x,z,s,w)$ is zero if and only if 
the \textsc{Subset Sum} instance is feasible.

We next express the objective function in the form required by
Problem~\eqref{eq:QP}, which does not include an additive constant. The only
constant terms produced by expanding $\Psi$ occur in
\eqref{eq:strong-reduction-objective-d}. Since $t_k\in\{0,1\}$, their sum
is
$$
    C:=\sum_{k=1}^{\ell}t_k^2
      =\sum_{k=1}^{\ell}t_k.
$$
Define
$F(x,z,s,w):=\Psi(x,z,s,w)-C$.
After expansion, $F$ can be written as
$$
    F(y)=y^\top Qy+c^\top y
$$
for a symmetric integral matrix $Q$ and an integral vector $c$, where $y \in [0,1]^N$
collects all the introduced variables. Note that $Q$ is integral because every product of two distinct variables arises from squaring an affine expression and therefore has an even coefficient in $F$. We then have
$$
    \min_{y\in[0,1]^N}F(y)=-C
$$
if and only if the original \textsc{Subset Sum} instance is feasible.
Thus, we define the decision threshold as
$\gamma:=-C$.

We now verify strong $\NP$-hardness. The constructed instance of Problem~\eqref{eq:QP} has
$N=2n\ell+n+\ell$
variables.
Since
$$
    \ell=\left\lceil\log_2(U+1)\right\rceil
    =O\left(\log n+\max_{i=1,\dots,n}\log(a_i+1)\right),
$$
the constructed instance has length polynomial in the length of the
\textsc{Subset Sum} instance.
Moreover, all numerical coefficients in the resulting quadratic optimization problem 
are bounded independently of the \textsc{Subset Sum} instance. To see this, observe that
each squared affine expression has coefficients in
$\{-1,0,1,2\}$,
and every variable occurs in only a constant number of such expressions.
A direct inspection of
\eqref{eq:strong-reduction-objective-a}--%
\eqref{eq:strong-reduction-objective-e} gives
$\|Q\|_{\max}\leq 5$ and $\|c\|_\infty\leq 4$.
Furthermore,
$|\gamma|=C\leq\ell\leq N$.
Thus every number appearing in the constructed instance is bounded by a polynomial in the number of variables, and the decision version of Problem~\eqref{eq:QP} is therefore strongly $\NP$-hard~\cite{GarJohBook}.

It remains to establish the treewidth bound. We describe a tree decomposition for a graph that contains the interaction graph of Problem~\eqref{eq:QP} as a subgraph. Since deleting edges cannot increase
treewidth, this suffices even if some quadratic coefficients cancel after expansion.
For each $i \in [n]$, define the bag
$$
A_i:=\{s_{i-1},s_i,z_{i\ell}\}\setminus\{s_0\}.
$$
Arrange these bags in the path
$A_1-A_2-\cdots-A_n$.
For each $i\in [n]$ and $k\in[\ell]$, define the bag
$$P_{ik}:=\{z_{i,k-1},z_{ik},x_{ik}\}\setminus\{z_{i0}\}.
$$
Moreover, for each $i \in [n]$ and $k \in [\ell-1]$, define the bag
$$R_{ik}:= \{z_{ik},x_{ik},x_{i,k+1}\}.
$$
For each fixed $i$, arrange these bags in the path
$P_{i1}-R_{i1}-P_{i2}-R_{i2}
    -\cdots-R_{i,\ell-1}-P_{i\ell}$,
and attach $P_{i\ell}$ to $A_i$. Define the bags
$$   
H_k:=\{w_{k-1},w_k\}\setminus\{w_0\}, \quad \forall k \in [\ell],
$$
and arrange these bags in the path
$H_1-H_2-\cdots-H_\ell$.
Finally, introduce the bag
$E:=\{s_n,w_\ell\}$,
and attach $E$ to both $A_n$ and $H_\ell$. First, it is simple to check that every variable $x,z,s,w$ appears in at least one of the bags defined above, implying that Property 1 of a tree decomposition is satisfied. Second,
the variables corresponding to every quadratic term in $\Psi$ are contained in at least one of these bags, implying that Property 2 of a tree decomposition is satisfied. 
It remains to verify Property~3. We do so by
explicitly identifying, for each variable, all bags in which it occurs. The following cases arise:

\medskip
\begin{itemize}[leftmargin=0.8cm]
\item [$(i)$] variables $x_{ik}$, $i \in [n]$, $k \in [\ell]$: 
if $k=1$ and $\ell\geq 2$, then $x_{i1}$ occurs in the bags $P_{i1}$ and $R_{i1}$, which are adjacent along a path. For
$2\leq k\leq \ell-1$, the variable $x_{ik}$ occurs in
$R_{i,k-1}$, $P_{ik}$ and $R_{ik}$.
These three bags occur consecutively along a path as
$R_{i,k-1}-P_{ik}-R_{ik}$.
Finally, $x_{i\ell}$ occurs exactly in
$R_{i,\ell-1}$ and $P_{i\ell}$, which are adjacent. If $\ell=1$, then $x_{i1}$ occurs
only in $P_{i1}$, and the claim is immediate.

\item [$(ii)$] variables $z_{ik}$, $i \in [n]$, $k \in [\ell]$: for $1\leq k\leq \ell-1$, 
$z_{ik}$ occurs in the three bags
$P_{ik}$, $R_{ik}$, and $P_{i,k+1}$, which
appear consecutively as
$P_{ik}-R_{ik}-P_{i,k+1}$.
The variable $z_{i\ell}$ occurs in
$P_{i\ell}$ and $A_i$.
By construction, $P_{i\ell}$ is adjacent to $A_i$.

\item [$(iii)$] variables $s_i$, $i \in [n]$:
for $i\in [n-1]$, the variable $s_i$ occurs in the two bags
$A_i$ and $A_{i+1}$, which are adjacent along the path
$A_1-A_2-\cdots-A_n$. The variable $s_n$ occurs
in $A_n$ and $E$
which are adjacent by construction. 

\item [$(iv)$] variables $w_k$, $k \in [\ell]$: for
$k\in [\ell-1]$, the variable $w_k$ occurs in
$H_k$ and
$H_{k+1}$, which are adjacent in the path
$H_1-H_2-\cdots-H_\ell$.
The variable $w_\ell$ occurs in
$H_\ell$ and $E$ which are adjacent by construction. 
\end{itemize}
\medskip

Hence property~3 of the tree decomposition holds.
Every bag has cardinality at
most three, and therefore the interaction graph has treewidth at most two and this completes the proof.
\end{proof}

\begin{remark}\label{rem:strong-vs-weak-hardness}
The machinery in the proof of
Theorem~\ref{thm:strong-NP-hardness-tw2} is needed to establish
\emph{strong} $\NP$-hardness. If one were interested only in
$\NP$-hardness, a substantially simpler construction would suffice.
Indeed, given a \textsc{Subset Sum} instance
$a_1,\ldots,a_n,T$, one may set $U:=\sum_{i=1}^n a_i$, introduce only
variables $x_i,s_i\in[0,1]$, with $s_0:=0$, and consider the problem of
minimizing
\[
    (Us_n-T)^2
    +\sum_{i=1}^n
        (Us_i-Us_{i-1}-a_i x_i)^2
    +\sum_{i=1}^n x_i(1-x_i).
\]
The objective function has minimum value zero if and only if the
\textsc{Subset Sum} instance is feasible, and its interaction graph has treewidth at most two, with bags
$\{s_{i-1},s_i,x_i\}$, $i \in [n]$. This construction, however, uses the integers $a_i$ and $U$ directly as coefficients and therefore establishes only weak $\NP$-hardness.
The bit-serial construction used in the proof of
Theorem~\ref{thm:strong-NP-hardness-tw2} replaces these large numerical
coefficients by polynomially many local relations with bounded
coefficients, thereby transferring the information contained in the
binary representations of the input integers from coefficient magnitude
to the combinatorial structure of the instance. This idea is closely
related to the bounded-coefficient encoding used by Eisenbrand
et al.~\cite{eisenbrand19} in their study of integer programming (Section~5.2 of~\cite{eisenbrand19}).

Cifuentes and Parrilo~\cite{CifPar16} show that checking the feasibility
of quadratic polynomial systems is $\NP$-hard even
when the interaction graph is a path, using a reduction from
\textsc{Subset Sum} (Example~1 in~\cite{CifPar16}). Their construction retains the potentially large
\textsc{Subset Sum} integers as coefficients and therefore yields only
weak $\NP$-hardness. Note that the technique used
in the proof of Theorem~\ref{thm:strong-NP-hardness-tw2} does not
directly strengthen their result while preserving treewidth at most one: each bit-serial relation induces a triangle in the
associated graph. Namely, the resulting construction has treewidth two. However, interestingly, the proof technique of Theorem~\ref{thm:quartic-path} can be used to prove that checking the feasibility
of quadratic polynomial systems is strongly $\NP$-hard even
when the interaction graph is a path.
There is also a close connection between our result and the construction of Bienstock and Mu\~noz~\cite{BieMun18}. Their construction yields weak $\NP$-hardness for a polynomial-feasibility problem whose interaction graph has treewidth two.
In fact, replacing the large coefficients in the Bienstock--Mu\~noz construction by the bit-serial encoding used in Theorem~\ref{thm:strong-NP-hardness-tw2}, we deduce that their  polynomial-feasibility problem is strongly $\NP$-hard.
\end{remark}

\section{Larger treewidths: A polynomial-time solvable class}
\label{sec: poly}

In this section, we obtain sufficient conditions for the polynomial-time solvability of Problem~\eqref{eq:QP} for classes of instances whose interaction graphs may have unbounded treewidth. The key idea is that there always exists an optimal solution to Problem~\eqref{eq:QP} in which every variable with a nonpositive quadratic coefficient is binary. This allows us to separate the problem into a combinatorial part and a continuous part, and to exploit the tractability of the corresponding components together with suitably simple interactions between them. We first introduce some terminology and tools that will be used in the proofs of our results.

We make use of the following celebrated result in discrete optimization that is concerned with tractability of binary optimization problems whose interaction graph has bounded treewidth. In the following, by $\poly(x,y)$ we mean a polynomial function in $x,y$.

\begin{proposition}[\cite{CraHanJau90}]\label{bpoly}
Consider a binary optimization problem
$$
\min_{x\in\{0,1\}^V}\sum_{\ell=1}^m f_\ell(x_{S_\ell}),
$$
where $S_\ell\subseteq V$ and $f_\ell:\{0,1\}^{S_\ell}\to\Q$ for every $\ell\in[m]$. Define the interaction graph $G=(V,E)$ by letting $\{i,j\}\in E$ whenever $i,j\in S_\ell$ for some $\ell\in[m]$. If $\tw(G)=\kappa$, then the problem can be solved by dynamic programming in $2^{O(\kappa)}\poly(|V|,m)$ operations, provided that the functions $f_\ell$ can be evaluated in polynomial time. In particular, if $\kappa\in O(\log\poly(|V|,m))$, then the problem can be solved in polynomial time.
\end{proposition}

\subsection{Hyperplane arrangements and the essential rank}

Hyperplane arrangements~\cite{Stanley07} are a useful geometric tool to partition the space into finitely many polyhedral regions. When the geometry of the hyperplanes is sufficiently low-dimensional, the number of such regions can be polynomial rather than exponential. We use this idea to reduce the number of cases that must be considered in our algorithm. In the following, we provide a brief overview of the results that we need to prove our result.

Let $\H=\{H_1,\dots,H_M\}$ be a collection of hyperplanes in $\R^d$, where $H_i=\{x\in\R^d:a_i^\top x=b_i\}$ and $a_i \neq 0$.
The \emph{arrangement} generated by $\H$, denoted by
$\A(\H)$, is the collection of connected components of
$$
\R^d\setminus \bigcup_{i=1}^M H_i.
$$
The elements of \(\mathcal A(\mathcal H)\) are called the \emph{cells} of the arrangement. For every cell $R\in\mathcal A(\mathcal H)$, the sign of
$a_i^\top x-b_i$
is constant on $R$ for every $i\in[M]$. Consequently,
$$
R=
\bigcap_{i=1}^M
\{x\in\R^d:\sigma_i(a_i^\top x-b_i)>0\}
$$
for suitable signs $\sigma_i\in\{-1,+1\}$. In particular, every cell is an open polyhedron.
Moreover, by Lemma~2.5 in~\cite{EdeRouSei86} the number of cells satisfies
\begin{equation}\label{cellnum}
|\mathcal A(\mathcal H)|
\le
\sum_{r=0}^d {M\choose r} \in O(M^d).
\end{equation}
By Theorem~3.3 in
\cite{EdeRouSei86}, an arrangement $\A(\H)$ of hyperplanes in $\mathbb R^d$ can be constructed in time $O(M^d)$. This implies that if $d \in O(1)$, then the cells of the arrangement can be computed in polynomial time.
The upper bound on the number of cells given by~\eqref{cellnum} can be improved using the concept of essential rank. 
The \emph{essential rank} of $\mathcal H$, denoted by $\rho$, is the dimension of the linear space spanned by the normals of the hyperplanes in the arrangement~\cite{Stanley07}. Define $W:={\rm span}\{a_i:i\in[M]\}$. We then have
$\rho:=\dim (W)$.
An arrangement is essential when $\rho$ is equal to the
dimension of the ambient space. It then follows that
$\mathcal H_W:=\{H\cap W:H\in\mathcal H\}$
is an essential arrangement in $W$. 
Stanley~\cite{Stanley07} shows that 
the number of cells of $\mathcal H$ is equal to the number of cells of $\H_W$. This together with~\eqref{cellnum} implies that
$$
|\A(\H)|=|\A(\H_W)|\le\sum_{r=0}^{\rho}{M\choose r}\in O(M^\rho).
$$

\subsection{A sufficient condition for polynomial-time solvability}
Let $G=(V,E)$ be a graph.
Recall that for a subset $S \subseteq V$, we denote by $G_S$ the subgraph of $G$ induced by $S$.
Consider a vertex $v \in V$. We define the \emph{neighborhood} of $v$ in $G$ as
\begin{equation}\label{nghood}
N_G(v) := \big\{u \in V : \{u,v\} \in E \big\}.
\end{equation}
Similarly, for a subset $S \subseteq V$, we define the neighborhood of $S$ in $G$ as: 
\begin{equation}\label{nghoodset}
N_G(S) := \Big(\bigcup_{v\in S} N_G(v)\Big) \setminus S.
\end{equation}

We are now ready to present the main result of this section. 
Figure~\ref{fig:oneBlock-example} demonstrates the application of the theorem.

\begin{figure}[t]
    \centering
    \scalebox{0.85}{\begin{tikzpicture}[
        vplus/.style={circle,draw,fill=black,minimum size=5.5pt,inner sep=0pt},
        nbr/.style={circle,draw,fill=black!30,minimum size=5.5pt,inner sep=0pt},
        rest/.style={circle,draw,fill=white,minimum size=5.5pt,inner sep=0pt}
    ]
        \node[vplus] (c) at (0,0) {};
        \foreach \j in {0,...,11}{
            \pgfmathsetmacro{\ang}{15+30*\j}
            \node[nbr]  (g\j) at (\ang:1.00) {};
            \node[rest] (m\j) at (\ang:1.85) {};
            \node[rest] (r\j) at (\ang:2.70) {};
            \draw (c) -- (g\j) -- (m\j) -- (r\j);
        }
        \draw (r11) -- (r0);
        \foreach \p in {1,3,5,7,9}{
            \pgfmathtruncatemacro{\q}{\p+1}
            \draw (r\p) -- (r\q);
        }
        \node[anchor=north] at (0,-3.05) {(a)};
    \end{tikzpicture}}
    \hspace{0.04\textwidth}
    \scalebox{0.85}{\begin{tikzpicture}[
        vplus/.style={circle,draw,fill=black,minimum size=5.5pt,inner sep=0pt},
        nbr/.style={circle,draw,fill=black!30,minimum size=5.5pt,inner sep=0pt},
        rest/.style={circle,draw,fill=white,minimum size=5.5pt,inner sep=0pt}
    ]
        \node[vplus] (c) at (0,0) {};
        \foreach \j in {0,...,11}{
            \pgfmathsetmacro{\ang}{15+30*\j}
            \node[nbr]  (g\j) at (\ang:1.00) {};
            \node[rest] (m\j) at (\ang:1.85) {};
            \node[rest] (r\j) at (\ang:2.70) {};
            \draw (c) -- (g\j) -- (m\j) -- (r\j);
        }
        \foreach \j in {0,...,10}{
            \pgfmathtruncatemacro{\q}{\j+1}
            \draw (r\j) -- (r\q);
        }
        \draw (r11) -- (r0);
        \node[anchor=north] at (0,-3.05) {(b)};
    \end{tikzpicture}}
    \caption{
    Two families of interaction graphs illustrating the application of
    Theorem~\ref{th:oneBlock}.
    In both cases, $V^+$ consists of a single black vertex, the vertices in
    $N_G(V^+)$ are depicted in gray, and the remaining vertices are depicted in white.
    Moreover, we assume that $|N_G(V^+)|\in\Theta(|V|)$.
    In~(a), each connected component
    $D$ of $G_{V^-}$ is a path. Therefore,
    $\tw(G_{V^-})=1$ and for each connected component
    $D$ we have $|D \cap N_G(V^+)|=2$. Thus all the assumptions of
    Theorem~\ref{th:oneBlock} are satisfied.
    In~(b), $G_{V^-}$ is connected with a
    unique connected component $D$, and therefore
    $|D\cap N_G(V^+)|=|N_G(V^+)|\in\Theta(|V|)$, implying that
    Theorem~\ref{th:oneBlock} does not
    apply. Part~(a) also satisfies the assumptions of Corollary~\ref{cor:fixVplus} because $|V^+|=1$.
    }
    \label{fig:oneBlock-example}
\end{figure}
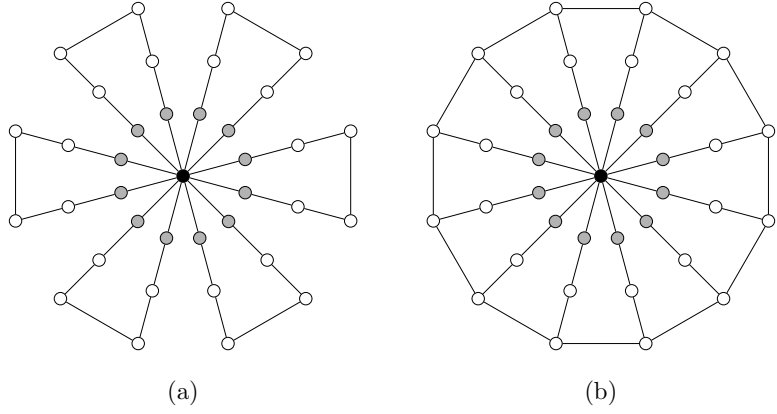
\begin{theorem}\label{th:oneBlock}
Let $G=(V,E)$ be the interaction graph of Problem~\eqref{eq:QP}.
Define $V^+:=\{i\in V:q_{ii}>0\}$, and $V^-:=V\setminus V^+$.
Suppose that the following conditions hold:
\medskip
\begin{enumerate}
\item
$\tw(G_{V^-})\in O(\log |V|)$ and for every connected component
$D$ of $G_{V^-}$, we have 
$$|D\cap N_G(V^+)|\in O(\log |V|).$$

\item
For every connected component $C$ of $G_{V^+}$ and every
$d\in\Q^C$, the optimization problem
\begin{equation*}
\min_{x_C\in[0,1]^C}
\left(
x_C^\top Q_Cx_C+d^\top x_C
\right)
\end{equation*}
can be solved in time polynomial in
$\L+\langle d\rangle$.

\item ${\rm rank}(Q_{V^+,N_G(V^+)})\in O(1)$.

\end{enumerate}
\medskip
Then Problem~\eqref{eq:QP} can be solved in time polynomial in the input length.
\end{theorem}

\begin{proof}
It can be checked that there exists an optimal solution $x^*$ of Problem~\eqref{eq:QP} with $x^*_i \in \{0,1\}$ for all $i \in V^-$. Therefore, it suffices to solve the following mixed-binary quadratic program: 
\begin{align}\label{mixedQP}
\min \quad & x^\top Q x + c^\top x \\
{\rm s.t.} \quad & 0 \leq x_i \leq 1, \quad \forall i \in V^+\nonumber\\
& x_i \in \{0,1\}, \quad \forall i \in V^-. \nonumber
\end{align}
Denote by $f(x)$ the objective function of Problem~\eqref{mixedQP}.
Let $D_1,\dots,D_p$
denote the connected components of $G_{V^-}$ and for each $k\in[p]$, define
$A_k:=D_k\cap N_G(V^+)$ and $F_k := D_k \setminus A_k$.
Since  $D_1,\dots,D_p$ are the connected components of $G_{V^-}$, there are no bilinear terms $x_i x_j $ in $f(x)$ with  $i \in D_k$, $j \in D_{\ell}$ with $k \neq \ell \in [p]$. Hence the objective function can be written as:
$$
f(x)= f_+(x_{V^+}) + \sum_{k=1}^p\Big(g_k(x_{D_k})+2\sum_{i\in A_k} x_i\, Q_{V^+,i}^\top x_{V^+}\Big),
$$
where
$$
f_+(x_{V^+}):=x_{V^+}^\top Q_{V^+}x_{V^+}+c_{V^+}^\top x_{V^+},
$$
and
$$
g_k(x_{D_k})
:=
x_{D_k}^\top Q_{D_k}x_{D_k}
+
c_{D_k}^\top x_{D_k}.
$$
Now fix $x_{V^+}\in[0,1]^{V^+}$. Then minimizing over the binary variables $x_{V^-}$ decomposes over the connected components $D_k$, $k \in [p]$:
$$
\min_{x_{V^-}\in\{0,1\}^{V^-}} f(x_{V^+},x_{V^-})
=f_+(x_{V^+}) + \sum_{k=1}^p \theta_k(x_{V^+}),
$$
where
$$
\theta_k(x_{V^+})
:=\min_{x_{D_k}\in\{0,1\}^{D_k}}
\Big(g_k(x_{D_k})+2\sum_{i\in A_k}x_i\,Q_{V^+,i}^\top x_{V^+}
\Big).
$$
Fix \(k\in[p]\). For each
$z\in\{0,1\}^{A_k}$, define
\begin{equation}\label{probaux1}
\mu_k(z)
:=
\min_{w\in\{0,1\}^{D_k\setminus A_k}}
g_k(z,w),
\end{equation}
where \(g_k(z,w)\) denotes the value of \(g_k(x_{D_k})\) after letting $x_{A_k}=z$ and $x_{D_k\setminus A_k}=w$.
More explicitly,
$$
g_k(z,w)
=
w^\top Q_{F_k}w
+
\left(
c_{F_k}+2Q_{F_k,A_k}z
\right)^\top w
+
z^\top Q_{A_k}z+c_{A_k}^\top z.
$$
Therefore, the interaction graph of the binary
quadratic optimization problem in~\eqref{probaux1} is $G_{F_k}$. Since $G_{F_k}$ is a subgraph of $G_{D_k}$, from assumption~1 it follows that $\tw(G_{F_k}) \leq \tw(G_{D_k}) \leq \tw(G_{V^-}) \in O(\log (|V|))$. Therefore, by Proposition~\ref{bpoly}, for every fixed $z\in\{0,1\}^{A_k}$, the value $\mu_k(z)$ can be computed in time polynomial in the input length. 
Moreover, since by assumption~1, we have
$|A_k| \in O(\log |V|)$, all values of $\mu_k(z)$, 
$z\in\{0,1\}^{A_k}$ can be computed in time polynomial in $|V|$. Finally, since $p \leq |V|$, values of $\mu_k(z)$, 
$z\in\{0,1\}^{A_k}$ for all $k \in [p]$ can be computed in polynomial time.

For each $z\in\{0,1\}^{A_k}$, define 
\begin{equation}\label{bkdef}
b_k(z):=2\sum_{i\in A_k} z_i Q_{V^+,i}.
\end{equation}
Then we have
$$
\theta_k(x_{V^+})
=
\min_{z\in\{0,1\}^{A_k}}
\left(
\mu_k(z)+b_k(z)^\top x_{V^+}
\right).
$$
Therefore, Problem~\eqref{mixedQP} is equivalent to the following optimization problem:
\begin{equation}\label{pready}
\min_{x_{V^+}\in[0,1]^{V^+}}
\left(
f_+(x_{V^+})
+
\sum_{k=1}^p
\min_{z\in\{0,1\}^{A_k}}
\left(
\mu_k(z)+b_k(z)^\top x_{V^+}
\right)
\right).
\end{equation}

So far, for $k \in [p]$, we denoted by $z$ a point in $\{0,1\}^{A_k}$.
From now on, it will be useful to keep explicitly the dependence on $k$ and we denote it instead by $z_k$.

\begin{claim}
\label{claim ADP}
We can compute in time polynomial in $|V|$ a set 
$$Z \subseteq \{0,1\}^{A_1} \times \cdots \times \{0,1\}^{A_p}$$ 
of cardinality polynomial in $|V|$
such that Problem~\eqref{pready} is equivalent to the following optimization problem:
\begin{equation}\label{pready 2}
\min_{(z_1,\dots,z_p)\in Z}
\left(
\min_{x_{V^+}\in[0,1]^{V^+}}
\left(
f_+(x_{V^+})
+
\sum_{k=1}^p
\left(
\mu_k(z_k)+b_k(z_k)^\top x_{V^+}
\right)
\right)
\right).
\end{equation}
\end{claim}

\begin{cpf}
The function $\theta_k(x_{V^+})$
is the minimum of finitely many affine functions of $x_{V^+}$. We partition $\R^{V^+}$
into regions over which, for each $k \in [p]$, the minimizing affine function is fixed. Such changes can occur only at points where two affine functions become
equal. Therefore, for each $k\in[p]$ and every distinct pair
$z_k,z_k'\in\{0,1\}^{A_k}$, we define the hyperplane
\begin{equation}\label{hyperplane}
H_{k,z_k,z_k'}
:=
\left\{
y\in\mathbb R^{|V^+|}:
\mu_k(z_k)+b_k(z_k)^\top y
=
\mu_k(z'_k)+b_k(z'_k)^\top y
\right\}.
\end{equation}
If $b_k(z_k)=b_k(z'_k)$ then $H_{k,z_k,z_k'}$ either defines the empty set or is all of $\R^{V^+}$, and therefore does not create a cell boundary. Hence, in the arrangement, we keep only the pairs for which $b_k(z_k) \neq b_k(z_k')$.
Thus, let
\begin{equation}\label{allhyperplanes}
\mathcal H
:=
\Big\{H_{k,z_k,z_k'}:
k\in[p],\ z_k,z_k'\in\{0,1\}^{A_k},\; z_k\neq z_k', \; b_k(z_k) \neq b_k(z'_k)\Big\}.
\end{equation}
Define
$$
L:=\operatorname{span}\Big\{Q_{V^+,i}:i\in N_G(V^+)\Big\}
\subseteq \mathbb R^{V^+}.
$$
Then 
$$\dim (L)={\rm rank}(Q_{V^+,N_G(V^+)})=:r.$$
For any $z_k \neq z_k'\in \{0,1\}^{A_k}$, from~\eqref{bkdef} it follows that
$$
b_k(z_k)-b_k(z'_k)= 2\sum_{i\in A_k}(z_i-z_i')Q_{V^+,i}.
$$
Since $A_k\subseteq N_G(V^+)$, it follows that
$b_k(z_k)-b_k(z'_k)\in L$. 
Thus, every hyperplane normal in the arrangement belongs to $L$, implying that the essential rank $\rho$ of $\H$ is upper bounded by $r$. Therefore, by Lemma~2.5 in~\cite{EdeRouSei86} and Section~1.1 of~\cite{Stanley07} it follows that
$$
|\mathcal A(\mathcal H)| \leq \sum_{i=0}^{r} {|\mathcal H|\choose i} \in O(|\mathcal H|^{r}).
$$
From~\eqref{hyperplane} and~\eqref{allhyperplanes}, it follows that $|\mathcal H| \leq \sum_{k=1}^p{\binom{2^{|A_k|}}{2}}$.
Since by assumption~1, we have
$|A_k| \in O(\log |V|)$ for all $k$, and $p\le |V|$, we deduce that
$|\mathcal H| \in |V|^{O(1)}$. Moreover, by assumption~3, we have
$r \in O(1)$. Therefore, 
the number of cells of $\mathcal A(\mathcal H)$ is polynomial in $|V|$. Let $W$ denote the linear space spanned by the normals of the hyperplanes in $\H$. Then $\dim(W)=\rho\leq r$. By the essentialization argument given above, the cells of $\A(\H)$ are in one-to-one correspondence with the cells of the essentialized arrangement
$\H_W:=\{H\cap W:H\in\H\}$.
Since $\rho\in O(1)$ and $|\H|\in |V|^{O(1)}$, by
Theorem~3.3 in~\cite{EdeRouSei86}, applied to the arrangement
$\H_W$ in the $\rho$-dimensional space $W$, these cells can be
computed in time polynomial in $|V|$.

Next, fix one cell $R\in\mathcal A(\H)$. By definition of the arrangement, for every $k\in[p]$, the ordering of the affine functions
$\mu_k(z_k)+b_k(z_k)^\top x_{V^+}$, $z\in\{0,1\}^{A_k}$,
is constant on $R$. Hence, for every $k\in[p]$, there exists
$\tilde z_k\in\{0,1\}^{A_k}$
such that
$$
\min_{z\in\{0,1\}^{A_k}}
\left(
\mu_k(z_k)+b_k(z_k)^\top x_{V^+}
\right)
=
\mu_k(\tilde z_k)+b_k(\tilde z_k)^\top x_{V^+}
\qquad
\forall x_{V^+}\in R.
$$
For each cell $R\in\mathcal A(\H)$ we then add $(\tilde z_1, \dots \tilde z_p) \in \{0,1\}^{A_1} \times \cdots \times \{0,1\}^{A_p}$ to the set $Z$. Since $|Z| \leq |\mathcal A(\H)|$, it follows that $|Z| \in |V|^{O(1)}$
and since $\rho \in O(1)$, the cells and hence $Z$ can be computed in polynomial time.
We now show that this construction of $Z$ yields the equivalence between Problems~\eqref{pready} and \eqref{pready 2}. Fix $x_{V^+}\in[0,1]^{V^+}$. Define
$$
\varphi(x_{V^+}):=f_+(x_{V^+})+\sum_{k=1}^p\theta_k(x_{V^+}),
$$
and
$$
\psi_{z}(x_{V^+}):=f_+(x_{V^+})+\sum_{k=1}^p\big(\mu_k(z_k)+b_k(z_k)^\top x_{V^+}\big),
$$
for $z=(z_1,\dots,z_p)\in Z$, so that $\varphi$ is the objective value of Problem~\eqref{pready} and $\psi_z$, for $z\in Z$, is the inner objective value of Problem~\eqref{pready 2}. 
For every $k\in[p]$ and every $z_k\in\{0,1\}^{A_k}$, by definition of $\theta_k$ we have
$$
\mu_k(z_k)+b_k(z_k)^\top x_{V^+}\geq \theta_k(x_{V^+}).
$$
Summing over $k$ and adding $f_+(x_{V^+})$ gives
$\psi_z(x_{V^+})\geq \varphi(x_{V^+})$
for all $z\in Z$. Denote by $\overline R$ the closure of the cell $R$. 
Since
$$[0,1]^{V^+} \subseteq
\bigcup_{R\in\mathcal A(\mathcal H)}\overline R,$$
there is a cell $R$ with $x_{V^+}\in\overline R$. Let
$\tilde z=(\tilde z_1,\dots,\tilde z_p)\in Z$ be the tuple added to $Z$ for this cell $R$, which by construction satisfies
$\theta_k(y)=\mu_k(\tilde z_k)+b_k(\tilde z_k)^\top y$
for all $y\in R$ and for all $k\in[p]$.
Both sides of this identity are continuous functions of $y$, so the identity extends from $R$ to its closure $\overline R$, and in particular it holds at $y=x_{V^+}$. Hence
$\varphi(x_{V^+})=\psi_{\tilde z}(x_{V^+})$.
Therefore, 
\begin{equation}
\varphi(x_{V^+})=\min_{z\in Z}\psi_z(x_{V^+})
\qquad
\forall\, x_{V^+}\in[0,1]^{V^+}.
\end{equation}
Taking the minimum with respect to $x_{V^+}\in[0,1]^{V^+}$ on both sides, and using that $Z$ is finite so that the minimum over $x_{V^+}$ and the minimum over $z\in Z$ can be interchanged, we obtain
$$
\min_{x_{V^+}\in[0,1]^{V^+}}\varphi(x_{V^+})
=
\min_{x_{V^+}\in[0,1]^{V^+}}\min_{z\in Z}\psi_z(x_{V^+})
=
\min_{z\in Z}\min_{x_{V^+}\in[0,1]^{V^+}}\psi_z(x_{V^+}),
$$
which is exactly the statement that Problems~\eqref{pready} and \eqref{pready 2} are equivalent.
\end{cpf}


Therefore, 
for every $z=(z_1,\dots,z_p) \in Z$, we solve the inner box-constrained quadratic program in \eqref{pready 2}, namely:
\begin{equation}\label{pready 2 inner}
\min_{x_{V^+}\in[0,1]^{V^+}}
\left(
f_+(x_{V^+})
+
\sum_{k=1}^p
\left(
\mu_k(z_k)+b_k(z_k)^\top x_{V^+}
\right)
\right),
\end{equation}
and we denote by $f^*_z$ its optimal value. 
For fixed $z\in Z$, define
$$
d_z:=c_{V^+}+\sum_{k=1}^p b_k(z_k),
\qquad
e_z:=\sum_{k=1}^p\mu_k(z_k).
$$
Then Problem~\eqref{pready 2 inner} can be written as
$$
e_z+
\min_{x_{V^+}\in[0,1]^{V^+}}
\left(
x_{V^+}^\top Q_{V^+}x_{V^+}
+
d_z^\top x_{V^+}
\right).
$$
Let $C_1,\dots,C_s$ denote the connected components of  $G_{V^+}$. Since there are no bilinear terms $x_ix_{i'}$ with $i\in C_j$ and $i'\in C_\ell$ for $j\neq\ell$, we have
$$
x_{V^+}^\top Q_{V^+}x_{V^+}
+
d_z^\top x_{V^+}
=
\sum_{j=1}^s
\left(
x_{C_j}^\top Q_{C_j}x_{C_j}
+
(d_z)_{C_j}^\top x_{C_j}
\right).
$$
Therefore, the optimal value $f^*_z$ of Problem~\eqref{pready 2 inner} is given by
$$
f^*_z
=
e_z+
\sum_{j=1}^s
\min_{x_{C_j}\in[0,1]^{C_j}}
\left(
x_{C_j}^\top Q_{C_j}x_{C_j}
+
(d_z)_{C_j}^\top x_{C_j}
\right).
$$
The length of $d_z$ is polynomial in the input length, since $d_z$ is obtained by summing polynomially many rational vectors whose entries are coefficients of the objective function. Hence, by assumption~2, each of the above box-constrained quadratic programs can be solved in time polynomial in the input length. Since $s\leq |V|$, it follows that $f^*_z$ can be computed in polynomial time for every $z\in Z$.
The optimal value $f^*$ of Problem~\eqref{pready} is then given by
$$
f^*=\min_{z\in Z}f^*_z.
$$
Since $|Z|\in |V|^{O(1)}$, we conclude that Problem~\eqref{eq:QP} can be solved in polynomial time in the input length.
\end{proof}

The three assumptions of Theorem~\ref{th:oneBlock} control three
different sources of complexity in Problem~\eqref{eq:QP}.
Assumption~1 concerns the \emph{combinatorial} part of the problem. The subgraph induced by the binary variables has logarithmic treewidth, and each of its connected components $D$ interacts with the continuous variables through only
$|D\cap N_G(V^+)|\in O(\log |V|)$ binary variables. Consequently, after fixing these interface variables, the remaining binary subproblem on each component can be solved in polynomial time. Assumption~2 concerns the \emph{continuous} part. Once the binary
choices are fixed, the problem over $V^+$ decomposes over the
connected components $C$ of $G_{V^+}$, and the box-constrained
quadratic problem associated with each such component must remain polynomial-time solvable under an arbitrary perturbation of its linear term. Finally, assumption~3 controls the \emph{interaction} between the binary and continuous parts. Although the binary variables may interact with many variables in $V^+$, the vectors describing these interactions span a space of constant dimension. This low-dimensional coupling is what allows the exponentially many
possible combinations of binary choices to be reduced, through the hyperplane-arrangement argument, to polynomially many relevant combinations.
Assumptions~1 and~3 of Theorem~\ref{th:oneBlock} are structural conditions that can be verified directly from the interaction graph and the matrix $Q$, respectively. In contrast, assumption~2 is less explicit, as it requires polynomial-time solvability of a family of box-constrained quadratic programs for every choice of the linear term. A simple sufficient
condition for assumption~2 is $Q_C\succeq0$ for every connected component $C$ of $G_{V^+}$, in which case each of these problems is convex. We next identify additional sufficient conditions under which assumption~2 holds. To
this end, we use the following result from~\cite{Aida26p}.

\begin{lemma}[\cite{Aida26p}]\label{lem:qp-kappa}
Consider Problem~\eqref{eq:QP} and let $\kappa \in \{0,\dots,n-1\}$ be the smallest integer such that every $(n-\kappa)\times(n-\kappa)$ principal submatrix of $Q$ is positive semidefinite. If no such integer exists, set $\kappa := n$. Then Problem~\eqref{eq:QP} can be solved in time
\begin{equation}\label{tottime}
O\Big((\kappa+1)\Big(\frac{2en}{\kappa}\Big)^\kappa\poly(\L)\Big)
\end{equation}
if $\kappa\geq1$, and in time $\poly(\L)$ if $\kappa=0$.
\end{lemma}

Notice that $\kappa = 0$ in Lemma~\ref{lem:qp-kappa} corresponds to the convex case, while if $\kappa \geq 1$, then Problem~\eqref{eq:QP} is not convex. Lemma~\ref{lem:qp-kappa} implies that if $\kappa \in O(1)$, then Problem~\eqref{eq:QP} can be solved in polynomial time.
We then obtain the following corollary of Theorem~\ref{th:oneBlock}.

\begin{corollary}\label{cor:simple}
Let $G=(V,E)$ be the interaction graph of Problem~\eqref{eq:QP}. Define
$V^+ := \{i\in V:q_{ii}>0\}$. 
Let $C_1,\dots,C_s$ denote the connected components of
$G_{V^+}$. For each $j\in[s]$, let $\kappa_j$ be the parameter defined
in Lemma~\ref{lem:qp-kappa} for the matrix $Q_{C_j}$.
Suppose that assumptions~1 and~3 of Theorem~\ref{th:oneBlock} are satisfied. Moreover, assume that for each $j\in[s]$, at least one of the following conditions holds:
\medskip
\begin{enumerate}[leftmargin=0.8cm]
\item [(a)] ${\rm rank}(Q_{C_j})\in O(1)$
\item [(b)] $G_{C_j}$ is a tree
\item [(c)] $\kappa_{j}\log\left(1+\frac{|C_j|}{\kappa_{j}}\right)
\in O(\log |V|)$,
where we define $0 \log(1+\frac{|C_j|}{0}) :=0$.
\end{enumerate}
\medskip
Then Problem~\eqref{eq:QP} can be solved in time polynomial in
the input length.
\end{corollary}

\begin{proof}
To prove the statement, it suffices to verify assumption~2 of Theorem~\ref{th:oneBlock}. Fix $j\in[s]$ and an arbitrary
$d\in\Q^{C_j}$, and consider the problem
\begin{equation}\label{eq:cor-simple-Cj}
\min_{x_{C_j}\in[0,1]^{C_j}}
\left(
x_{C_j}^\top Q_{C_j}x_{C_j}
+
d^\top x_{C_j}
\right).
\end{equation}
If part~(a) holds; \ie ${\rm rank}(Q_{C_j})\in O(1)$, then by Theorem~1 in~\cite{hladik21}, Problem~\eqref{eq:cor-simple-Cj} can be solved in polynomial time.
If part~(b) holds; \ie $G_{C_j}$ is a tree, then by Theorem~\ref{thm:main}, Problem~\eqref{eq:cor-simple-Cj} can be solved in polynomial time.
Suppose now that condition~(c) holds. Notice that $\kappa_j$ depends only on the quadratic matrix $Q_{C_j}$ and is therefore unchanged
when the linear coefficient is replaced by the arbitrary vector $d$. If $\kappa_j=0$, Lemma~\ref{lem:qp-kappa} implies that Problem~\eqref{eq:cor-simple-Cj} can be solved in time polynomial in
$\L+\langle d\rangle$. Suppose that $\kappa_j\geq1$. By Lemma~\ref{lem:qp-kappa},
Problem~\eqref{eq:cor-simple-Cj} can be solved in time
$$
O\left(
(\kappa_j+1)
\left(\frac{2e|C_j|}{\kappa_j}\right)^{\kappa_j}
\poly(\L+\langle d\rangle)
\right).
$$
Since $\kappa_j\leq |C_j|$, we have
$\kappa_j+1
\leq 2^{\kappa_j}
\leq
\left(1+\frac{|C_j|}{\kappa_j}\right)^{\kappa_j}$.
Moreover, setting $t:=|C_j|/\kappa_j\geq1$, we have
$2et\leq 8t\leq(1+t)^3$, and hence
$\left(\frac{2e|C_j|}{\kappa_j}\right)^{\kappa_j}
\leq
\left(1+\frac{|C_j|}{\kappa_j}\right)^{3\kappa_j}$.
Therefore,
$$
(\kappa_j+1)
\left(\frac{2e|C_j|}{\kappa_j}\right)^{\kappa_j}
\leq
\left(1+\frac{|C_j|}{\kappa_j}\right)^{4\kappa_j}.
$$
By condition~(c),
$\left(1+\frac{|C_j|}{\kappa_j}\right)^{4\kappa_j}
\in |V|^{O(1)}$.
It then follows that for every $j\in[s]$ and every $d\in\Q^{C_j}$, Problem~\eqref{eq:cor-simple-Cj} can be solved in time
polynomial in $\L+\langle d\rangle$, implying that assumption~2 of Theorem~\ref{th:oneBlock} is satisfied and this completes the proof.
\end{proof}

Consider condition~(c) of Corollary~\ref{cor:simple}.
As detailed in~\cite{Aida26p}, this condition has a natural interpretation in terms of the value of $\kappa_j$ relative to the size of $C_j$. In particular, condition~(c) is automatically satisfied if
$|C_j|\in O(\log |V|)$, regardless of the value of $\kappa_j$,
and it is also automatically satisfied if $\kappa_j\in O(1)$ regardless of the size of $|C_j|$. 

Clearly, assumption~3 of Theorem~\ref{th:oneBlock} depends on the matrix $Q$. We next obtain a sufficient condition that is independent of the coefficients of the objective function.
Let $H$ be the bipartite graph with bipartition $V^+$ and $N_G(V^+)$ and edge set
\begin{equation}\label{rankedge}
E(H):=\big\{\{i,j\}\in E:i\in V^+,\;j\in N_G(V^+)\big\}.
\end{equation}
Thus, the edges of $H$ correspond to the nonzero entries of $Q_{V^+,N_G(V^+)}$. Let $\nu(H)$ denote the cardinality of a maximum matching of $H$. By Theorem~1 in~\cite{Edmonds67}, we have
$$
{\rm rank}(Q_{V^+,N_G(V^+)})\leq\nu(H).
$$
Since $H$ is bipartite, K\"onig's theorem
(see Theorem~2.1.1 in~\cite{Diestel17}) implies $\nu(H)=\tau(H)$, where $\tau(H)$ denotes the minimum cardinality of a vertex cover of $H$.
Therefore, we obtain the following result:

\begin{corollary}\label{cor:crossMatching}
Let $G=(V,E)$ be the interaction graph of Problem~\eqref{eq:QP}.
Define $V^+:=\{i\in V:q_{ii}>0\}$.
Let $H$ be the bipartite graph with bipartition
$V^+$ and $N_G(V^+)$ and with the edge set defined by~\eqref{rankedge}. Suppose that assumptions~1 and~2 of Theorem~\ref{th:oneBlock} are satisfied. Moreover, suppose that the minimum cardinality $\tau(H)$ of a vertex cover of $H$ satisfies $\tau(H) \in O(1)$.
Then Problem~\eqref{eq:QP} can be solved in time polynomial in the input length.
\end{corollary}

By construction we have 
$$
{\rm rank}(Q_{V^+,N_G(V^+)}) \leq |V^+|.
$$ 
Therefore, in the special case where $|V^+| \in O(1)$, the assumptions of Theorem~\ref{th:oneBlock} can be simplified to obtain the following result:

\begin{corollary}\label{cor:fixVplus}
Let $G=(V,E)$ be the interaction graph of Problem~\eqref{eq:QP}. Define
$V^+ := \{i\in V:q_{ii}>0\}$.
Suppose that $|V^+|\in O(1)$ and that assumption~1 of Theorem~\ref{th:oneBlock} is satisfied.
Then Problem~\eqref{eq:QP} can be solved in time polynomial in
the input length.
\end{corollary}

\begin{proof}
It suffices to show that assumptions~2 and~3 of 
Theorem~\ref{th:oneBlock} are satisfied.
Since $|V^+|\in O(1)$, every connected component $C$ of
$G_{V^+}$ has constant size. Hence, by Lemma~\ref{lem:qp-kappa} the box-constrained quadratic
problem associated with $C$ can be solved in polynomial time for
every linear coefficient, and assumption~2 of
Theorem~\ref{th:oneBlock} is satisfied. Finally,
${\rm rank}(Q_{V^+,N_G(V^+)})\leq |V^+|\in O(1)$,
so assumption~3 is also satisfied. The result then follows.
\end{proof}

The following results are concerned with special cases where the vertex sets $V^+$ and $V^-$ are highly structured.

\begin{corollary}\label{cor:vertexCover}
Let $G=(V,E)$ be the interaction graph of Problem~\eqref{eq:QP}.
Suppose that $V^+ := \{i\in V:q_{ii}>0\}$ is a vertex cover of $G$. Suppose that assumptions~2 and~3 of Theorem~\ref{th:oneBlock} are satisfied.
Then Problem~\eqref{eq:QP} can be solved in time polynomial in
the input length.
\end{corollary}

\begin{proof}
It suffices to show assumption~1 of Theorem~\ref{th:oneBlock} holds.
Since $V^+$ is a vertex cover, $V^-:=V \setminus V^+$ is a stable set. Hence, every connected component $D$ of $G_{V^-}$ consists of a single vertex, implying that $\tw(G_{V^-})=0$ and 
$|D \cap N_G(V^+)|\leq1$ for every connected component $D$ of $G_{V^-}$. 
\end{proof}

\begin{corollary}\label{cor:VplusStable}
Let $G=(V,E)$ be the interaction graph of Problem~\eqref{eq:QP}.
Suppose that $V^+:=\{i\in V:q_{ii}>0\}$ is a
stable set of $G$. Suppose that assumptions~1 and~3 of Theorem~\ref{th:oneBlock} are satisfied.
Then Problem~\eqref{eq:QP} can be solved in time polynomial in
the input length.
\end{corollary}

\begin{proof}
It suffices to show that assumption~2 of Theorem~\ref{th:oneBlock} is satisfied.
Since $V^+$ is a stable set, every connected component $C$ of $G_{V^+}$ consists of a single vertex. The result then follows.
\end{proof}

Consider a bipartite graph $G$ with bipartition
$V^+,V^-$. Then both $V^+$ and $V^-$ are
stable sets, and every connected component of both
$G_{V^+}$ and $G_{V^-}$ consists of a single vertex.
Therefore, the following result is immediate from
Corollary~\ref{cor:simple}.

\begin{corollary}\label{cor:bipartite}
Let $G=(V,E)$ be the interaction graph of Problem~\eqref{eq:QP}. Define $V^+ := \{i\in V:q_{ii}>0\}$ and $V^-:= V \setminus V^+$. Suppose that $G$ is bipartite with bipartition $V^+, V^-$ and
${\rm rank}(Q_{V^+,N_G(V^+)})\in O(1)$.
Then Problem~\eqref{eq:QP} can be solved in time polynomial in the input length.
\end{corollary}

The next result considers the case where every vertex in $V^-$ has degree at most two in $G$.

\begin{corollary}\label{cor:atmost2}
Let $G=(V,E)$ be the interaction graph of Problem~\eqref{eq:QP}. Define $V^+ := \{i\in V:q_{ii}>0\}$ and $V^-:= V \setminus V^+$.
Suppose that every vertex in $V^-$ has degree at most two in $G$.
Moreover, assume that assumptions~2 and~3 of Theorem~\ref{th:oneBlock} are satisfied.
Then Problem~\eqref{eq:QP} can be solved in time polynomial in
the input length.
\end{corollary}

\begin{proof}
It suffices to show that assumption~1 of Theorem~\ref{th:oneBlock} is satisfied.
Let $D_1,\dots,D_p$ denote the connected components of $G_{V^-}$ and as before,
for each $k\in[p]$, define
$A_k:=D_k\cap N_G(V^+)$.
Since every vertex in $V^-$ has degree at most two in $G$, every connected
component $D_k$ of $G_{V^-}$ is either a path, a cycle, or an isolated vertex.
If $D_k$ is a cycle, then every vertex of $D_k$ has two neighbors in
$D_k$. Therefore, no vertex of $D_k$ can be adjacent to a vertex in $V^+$, since otherwise its degree in $G$ would be at least three. Hence, in this case $A_k=\emptyset$.
If $D_k$ is a path, then only the endpoints of the path can be adjacent to vertices in $V^+$. Indeed, every internal vertex of the path already has two neighbors in $D_k$ and therefore cannot have an additional neighbor in $V^+$. Hence, in this case $|A_k|\leq 2$. Finally, if $D_k$ is an isolated vertex, then $|A_k|\leq 1$.
Therefore, $|A_k|\leq2$ for all $k\in[p]$.
Moreover, since every $D_k$ is a path, a cycle, or an isolated vertex, we have
$\tw(G_{V^-})\leq 2$.
Therefore, assumption~1 of Theorem~\ref{th:oneBlock} is satisfied and this completes the proof.
\end{proof}

The following example demonstrated that our sufficient conditions establish polynomial-time solvability of some class of nonconvex box-constrained quadratic programs for which the Hessian is full-row rank and the treewidth of the interaction graph is unbounded. 

\begin{example}
For each $m \geq 1$, consider an instance of Problem~\eqref{eq:QP} with
$$
Q=
\begin{pmatrix}
I_m & J_m\\
J_m & -I_m
\end{pmatrix},
$$
where
$I_m$ is an $m \times m$ identity matrix, $J_m$ is an $m \times m$ matrix of all ones, and $c$ is an arbitrary rational vector.
It can be checked that $Q$ is indefinite and has full-row rank $2m$.
The interaction graph in this case is the complete bipartite graph $K_{m,m}$ with bipartition $V^+$ and $V^-$. We then have $\tw(K_{m,m}) = m$. Moreover, $Q_{V^+,V^-} = J_m$, implying that 
${\rm rank}(Q_{V^+,V^-})=1$.
Therefore, all assumptions of Corollary~\ref{cor:bipartite} are satisfied and hence, Problem~\eqref{eq:QP} can be solved in polynomial time in the input length. 
\end{example}

\subsection{A more general scheme}
In~\cite{Aida26p}, the author obtains a sufficient condition under which Problem~\eqref{eq:QP} can be solved in polynomial time. Next we state this result in a slightly more general form. The proof however follows directly from the proof of Theorem~1 in~\cite{Aida26p}.

\begin{theorem}[\cite{Aida26p}]\label{th:smallNeighborhood}
Let $G=(V,E)$ be the interaction graph of Problem~\eqref{eq:QP}. Define
$V^+:=\{i\in V:q_{ii}>0\}$ and $V^-:= V \setminus V^+$. Let
$C_1,\dots,C_s$ denote the connected components of $G_{V^+}$.
Let $\bar G$ be the graph on $V^-$ obtained from
$G_{V^-}$ by making $N_G(C_j)$ a clique for every
$j\in[s]$. Suppose that the following conditions hold:
\medskip
\begin{enumerate}
\item $\tw(\bar G)\in O(\log |V|)$.

\item For each $j\in[s]$ and every $d\in\Q^{C_j}$, 
problem~\eqref{eq:cor-simple-Cj}
can be solved in time polynomial in $\L+\langle d\rangle$.
\end{enumerate}
\medskip
Then Problem~\eqref{eq:QP} can be solved in time polynomial
in the input length.
\end{theorem}

\begin{remark}\label{rem:incomparable}
Theorems~\ref{th:oneBlock} and~\ref{th:smallNeighborhood} are
incomparable.
First, assumption~2 of Theorem~\ref{th:oneBlock} and that of Theorem~\ref{th:smallNeighborhood} are identical. Second, assumption~1 of Theorem~\ref{th:smallNeighborhood} implies that $|N_G(C_j)| \in O(\log |V|)$ for all $j \in [s]$. To see this, observe that each $N_G(C_j)$ is a clique in $\bar G$ and hence we have $|N_G(C_j)| \leq \tw(\bar G)+1 \in O(\log |V|)$.
That is, all continuous components have \emph{small neighborhoods}, an assumption that does not need to hold in Theorem~\ref{th:oneBlock}. To see this, for each $m\geq1$, let
$V^+=\{u\}$ and  $V^-=\{v_1,\dots,v_m\}$. Assume that
$V^-$ is a stable set and $u$ is adjacent to every vertex in
$V^-$. The connected components of $G_{V^-}$ are the singletons $\{v_i\}$, and hence
$\tw(G_{V^-})=0$ and $|D\cap N_G(V^+)|=1$
for every connected component $D$ of $G_{V^-}$. The unique connected component of $G_{V^+}$ is the singleton $\{u\}$, and
${\rm rank}(Q_{V^+,N_G(V^+)})\leq 1$.
Thus, all assumptions of Theorem~\ref{th:oneBlock} are satisfied. Now consider Theorem~\ref{th:smallNeighborhood}.
We have
$N_G(V^+)=N_G(u)=V^-$, implying that $\bar G=K_m$.
Therefore, $\tw(\bar G)=m-1$,
so assumption~1 of Theorem~\ref{th:smallNeighborhood} is violated.
Consequently, Theorem~\ref{th:oneBlock} does not follow from
Theorem~\ref{th:smallNeighborhood}.
Conversely, for each $m\geq1$, let
$V^+:=\{p_1,\dots,p_m\}$ and $V^-:=\{z_1,\dots,z_{m+1}\}$.
Let $G_{V^-}$ be the path
$z_1-z_2-\cdots-z_{m+1}$,
let $V^+$ be a stable set, and let $p_i$ be adjacent to
$z_i$ and $z_{i+1}$ for every $i\in[m]$. The connected components of $G_{V^+}$ are the singletons
$C_i:=\{p_i\}$, and
$N_G(C_i)=\{z_i,z_{i+1}\}$.
Since $\{z_i,z_{i+1}\}$ is already an edge of $G_{V^-}$, making $N_G(C_i)$ a clique adds no edge to $G_{V^-}$. Hence
$\bar G=G_{V^-}$,
and therefore $\tw(\bar G)=1$. Moreover, each $C_i$ is a singleton, so assumption~2 of Theorem~\ref{th:smallNeighborhood} is satisfied.
Thus, Theorem~\ref{th:smallNeighborhood} applies.
On the other hand, $G_{V^-}$ is connected, so it has a single
connected component $D=V^-$. Moreover, every vertex of $V^-$ is
adjacent to a vertex in $V^+$, and hence
$D\cap N_G(V^+)=V^-$.
Therefore,
$|D\cap N_G(V^+)|=m+1=\Theta(|V|)$,
and assumption~1 of Theorem~\ref{th:oneBlock} is violated.
Consequently, Theorem~\ref{th:smallNeighborhood} does not follow from Theorem~\ref{th:oneBlock}.
\end{remark}

The proof techniques of Theorems~\ref{th:oneBlock} and~\ref{th:smallNeighborhood} are fundamentally different.
To prove Theorem~\ref{th:oneBlock}, we first eliminate each binary component $D_k\setminus A_k$, $k \in [p]$ and subsequently add a clique on $A_k$. This is a polynomial-time operation thanks to the two key facts in assumption~1. In contrast, to prove Theorem~\ref{th:smallNeighborhood}, the author first eliminates each continuous component $C_j$, $j \in [s]$ and subsequently adds a clique $N_G(C_j)$ to graph $G$. This step can be performed in polynomial time, because from assumption~1 of Theorem~\ref{th:smallNeighborhood}, it follows that $|N_G(C_j)| \in O(\log |V|)$ for all $j \in [s]$. The incomparability of
Theorems~\ref{th:oneBlock} and~\ref{th:smallNeighborhood}
suggests combining the two elimination schemes. Namely, for an instance of Problem~\eqref{eq:QP} that does not satisfy the assumptions of either theorem, we first eliminate a subset of continuous components $C_j$ of $G_{V^+}$ with $|N_G(C_j)|\in O(\log |V|)$ using the technique of Theorem~\ref{th:smallNeighborhood} and
then apply Theorem~\ref{th:oneBlock} to the remaining graph.





\begin{theorem}\label{th:oneSwitch}
Let $G=(V,E)$ be the interaction graph of Problem~\eqref{eq:QP}.
Define $V^+:=\{i\in V:q_{ii}>0\}$ and $V^-:=V\setminus V^+$. Let
$C_1,\dots,C_s$ denote the connected components of $G_{V^+}$.
Suppose that the following conditions hold:
\medskip
\begin{enumerate}

\item There exists a subset $J\subseteq[s]$ such that, letting
$R:=V^+\setminus\bigcup_{j\in J}C_j$
and letting $\widehat G$ be the graph on $V^-$ obtained from
$G_{V^-}$ by making $N_G(C_j)$ a clique for every $j\in J$, we have:
\medskip
\begin{itemize}
\item [(i)]$\tw(\widehat G)\in O(\log |V|)$, and 
$|D\cap N_G(R)|\in O(\log |V|)$
for every connected component $D$ of $\widehat G$
\item [(ii)] ${\rm rank}\bigl(Q_{R,N_G(R)}\bigr)\in O(1)$.
\end{itemize}
\medskip

\item For every $j\in[s]$ and every $d\in\Q^{C_j}$, the optimization
problem~\eqref{eq:cor-simple-Cj} can be solved in time polynomial in $\L+\langle d\rangle$.

\end{enumerate}
\medskip
Then Problem~\eqref{eq:QP} can be solved in time polynomial in the
input length.
\end{theorem}

\begin{proof}
We first eliminate the continuous components $C_j$, $j\in J$, using
the proof technique of Theorem~1 in~\cite{Aida26p}. Since
$N_G(C_j)$ is a clique in $\widehat G$, assumption~1 implies that
$|N_G(C_j)|\in O(\log |V|)$ for every $j\in J$. Therefore, by
assumption~2, all such components can be eliminated in polynomial
time. We are then left with an  optimization problem of
the form
$$
\min_{\substack{x_R\in[0,1]^R\\x_{V^-}\in\{0,1\}^{V^-}}}
\Big\{
x_R^\top Q_Rx_R+c_R^\top x_R
+g(x_{V^-})
+2\sum_{i\in N_G(R)}x_iQ_{R,i}^\top x_R
\Big\},
$$
where $g$ is a binary objective function whose interaction graph is
$\widehat G$.
We now apply the proof technique of Theorem~\ref{th:oneBlock}, with
$R$ in place of $V^+$. The only difference is that the binary
objective $g(x_{V^-})$ need not be quadratic. However, when the variables
in $D\cap N_G(R)$ are fixed, for a connected component $D$ of
$\widehat G$, the remaining binary subproblem has interaction graph
contained in $\widehat G_D$. Hence, by assumption~1 and
Proposition~\ref{bpoly}, each such subproblem can still be solved
in polynomial time. Since
$|D\cap N_G(R)|\in O(\log |V|)$, all required values can be
computed in polynomial time.
The remainder of the proof of Theorem~\ref{th:oneBlock} applies
unchanged: condition~1(ii) yields the required constant-rank bound for
the hyperplane-arrangement argument, while assumption~2 guarantees
polynomial-time solvability of the remaining continuous components
under arbitrary linear perturbations. Therefore,
Problem~\eqref{eq:QP} can be solved in polynomial time.
\end{proof}

Theorems~\ref{th:oneBlock} and~\ref{th:smallNeighborhood} are special cases of Theorem~\ref{th:oneSwitch}.
Namely, letting $J=\emptyset$, we get $\widehat G=G_{V^-}$ and $R=V^+$; therefore,
Theorem~\ref{th:oneSwitch} reduces to
Theorem~\ref{th:oneBlock}. Alternatively, letting $J=[s]$, we get $R=\emptyset$ and $\widehat G = \bar G$. So assumption~(i) reduces
to $\tw(\bar G)\in O(\log |V|)$ and assumption~(ii) becomes vacuous. We thus recover
Theorem~\ref{th:smallNeighborhood}. Notice that the key to the power of Theorem~\ref{th:oneSwitch} is a careful selection of the subset $J$.
A simple rule is to let $J$ consist of the index set of all continuous components with a logarithmic-size neighborhood.  The following example shows that this rule is not sufficient. Namely, in general, among the continuous components having logarithmic-size neighborhoods, it may be necessary to eliminate only a proper subset in the first phase. 

\begin{example}\label{ex:proper-subset}
Let $G=(V,E)$ be the interaction graph of Problem~\eqref{eq:QP}.
For each integer $m\geq4$, let $\ell:=\lceil\log_2 m\rceil$.
Let $W:=\{w_1,\dots,w_m\}$ and suppose that $G_W$ is the path
$w_1-w_2-\cdots-w_m$. Moreover, for each $i\in[\ell]$, let
$D_i:=\{v_{i1},\dots,v_{i\ell}\}$ and suppose that $G_{D_i}$ is the path
$v_{i1}-v_{i2}-\cdots-v_{i\ell}$.
Assume that $W,D_1,\dots,D_\ell$ are pairwise disjoint and define
$V^-:= V \setminus V^+ = W\cup D_1\cup\cdots\cup D_\ell$.
Moreover, suppose that $W,D_1,\dots,D_\ell$ are the connected
components of $G_{V^-}$.
Define
$V^+:=\{u,c\}\cup\{b_1,\dots,b_m\}$,
and suppose that $V^+$ is a stable set. In addition to the edges of
$G_{V^-}$ specified above, suppose that $G$ contains exactly the following edges: $\{b_j,w_j\}$ for every $j\in[m]$, $\{u,w_1\}$, $\{u,v_{ij}\}$ for every $i,j\in[\ell]$, and
$\{c,v_{i1}\}$ for every $i\in[\ell]$.
We then have
$|V|=2m+\ell^2+2=\Theta(m)$, 
implying that $\ell=\Theta(\log |V|)$, while
$m\notin O(\log |V|)$.
We first show that neither Theorem~\ref{th:oneBlock} nor
Theorem~\ref{th:smallNeighborhood} applies to this instance.
Consider first Theorem~\ref{th:oneBlock}. Recall that 
$W=\{w_1,\dots,w_m\}$ is a connected component of $G_{V^-}$
and every $w_j$ is adjacent to $b_j\in V^+$. Hence
$W\cap N_G(V^+)=W$,
and therefore $|W\cap N_G(V^+)|=m\notin O(\log |V|)$.
Thus assumption~1 of Theorem~\ref{th:oneBlock} is violated.
Next consider Theorem~\ref{th:smallNeighborhood}; By construction,
$N_G(u)=\{w_1\}\cup D_1\cup\cdots\cup D_\ell$,
and hence
$|N_G(u)|=1+\ell^2$.
In the graph $\bar G$, the set $N_G(u)$ is a clique.
Therefore,
$\tw(\bar G)\geq |N_G(u)|-1=\ell^2.$
Since $\ell=\Theta(\log |V|)$, we have
$\ell^2\notin O(\log |V|)$, and hence
Theorem~\ref{th:smallNeighborhood} does not apply.

We now consider the strategy of first eliminating every connected component $C$ of $G_{V^+}$ satisfying
$|N_G(C)|\in O(\log |V|)$ and then applying
Theorem~\ref{th:oneBlock}. Since $V^+$ is a stable set, its connected components are
$\{u\},\{c\},\{b_1\},\dots,\{b_m\}$.
For every $j\in[m]$,
$N_G(b_j)=\{w_j\}$, while
$N_G(c)=\{v_{11},v_{21},\dots,v_{\ell1}\}$.
Hence $|N_G(b_j)|=1$ and $|N_G(c)|=\ell=O(\log |V|)$.
On the other hand,
$|N_G(u)|=1+\ell^2\notin O(\log |V|)$.
Thus this strategy eliminates $\{b_j\}$,
$j\in[m]$, together with $\{c\}$, while retaining $\{u\}$.
Eliminating $\{b_j\}$ creates no interaction between distinct binary variables, since $N_G(b_j)=\{w_j\}$. We next consider the elimination of $c$. In this case,
$\{v_{11},\dots,v_{\ell1}\}$ becomes a clique in the resulting graph $\bar G$ and therefore the components
$D_1,\dots,D_\ell$ are joined into a single connected component
$D^*:=D_1\cup\cdots\cup D_\ell$.
Since $u$ is adjacent to every vertex in $D^*$,
we have $|D^*\cap N_G(u)|=|D^*|
=\ell^2=\Theta((\log |V|)^2)$.
Thus the resulting instance still violates assumption~1 of
Theorem~\ref{th:oneBlock}.

Finally, consider eliminating only the collection
$\mathcal E:=\{\{b_j\}:j\in[m]\}$.
Since each $b_j$ has the single neighbor $w_j$, eliminating these
components creates no interaction between distinct binary variables.
Hence the connected components of the resulting binary interaction graph remain
$W,D_1,\dots,D_\ell$.
We then have
$|W\cap N_G(\{u,c\})|=1$ and
$|D_i\cap N_G(\{u,c\})|=|D_i|=\ell
=O(\log |V|)$ for every $i\in[\ell]$.
Moreover, $G_W$ and $G_{D_i}$, $i \in [\ell]$ are paths, and hence $\tw(G_{V^-})=1$. Thus assumption~1 of
Theorem~\ref{th:oneBlock} is satisfied.
The retained positive components $\{u\}$ and $\{c\}$ are singletons, so assumption~2 of Theorem~\ref{th:oneBlock} is satisfied. Finally,
${\rm rank}\left(Q_{\{u,c\},N_G(\{u,c\})}
\right)\leq2$,
and hence assumption~3 is satisfied as well.
Thus, eliminating all continuous components having
logarithmic-size neighborhoods does not suffice, whereas eliminating the proper subset $\mathcal E$ and then applying Theorem~\ref{th:oneBlock} yields a polynomial-time algorithm for solving Problem~\eqref{eq:QP}.
\end{example}





The next proposition indicates that the subset $J$ whose existence is required in condition~1 of
Theorem~\ref{th:oneSwitch} can be constructed in polynomial time.


\begin{proposition}\label{prop:findJ}
Under the notation of Theorem~\ref{th:oneSwitch}, whenever
condition~1 of that theorem holds, a subset $J\subseteq[s]$
satisfying condition~1 can be found in time polynomial in the
input length.
\end{proposition}

\begin{proof}
Let $n:=|V|$. Since condition~1 of Theorem~\ref{th:oneSwitch} holds, there exist
constants $\alpha,\beta>0$ and $r\in\mathbb Z_{\geq0}$, independent
of $n$, and a subset $J^*\subseteq[s]$ such that, with
$R^*:=V^+\setminus\cup_{j\in J^*}C_j$ and with $\widehat G^*$
defined as in Theorem~\ref{th:oneSwitch}, we have
$\tw(\widehat G^*)\leq\alpha\log n$,
$|D\cap N_G(R^*)|\leq\beta\log n$ for every connected component
$D$ of $\widehat G^*$, and
${\rm rank}(Q_{R^*,N_G(R^*)})\leq r$.
Set $k:=\lceil\alpha\log n\rceil$ and
$h:=\lceil\beta\log n\rceil$.
For every $j\in[s]$, define
$$
L_j:=\operatorname{span}\{Q_{v,V^-}:v\in C_j\}.
$$
Moreover, for every $I\subseteq[s]$ with $|I|\le r$, define
$$
L_I:=\operatorname{span}\Big\{Q_{v,V^-}:v\in \bigcup_{i \in I} C_i\Big\},
$$
where we define $L_\emptyset:=\{0\}$. We consider only  subsets $I$ for which $\dim(L_I)\le r$. For each such $I$, define
$$J_I:=
\{j\in[s]:L_j\not\subseteq L_I\},
$$
and let
$$
R_I:=V^+\setminus\bigcup_{j\in J_I}C_j.
$$
Let $\widehat G_I$ be the graph on $V^-$ obtained from $G_{V^-}$
by making $N_G(C_j)$ a clique for every $j\in J_I$.
By the definition of $J_I$, if $C_j\subseteq R_I$, then
$L_j\subseteq L_I$. Hence
$\operatorname{span}\{Q_{v,V^-}:v\in R_I\}
\subseteq L_I$.
Conversely, for every $i\in I$ we have $L_i\subseteq L_I$, and
therefore $i\notin J_I$. Thus $C_i\subseteq R_I$ for every
$i\in I$, which implies
$L_I \subseteq
\operatorname{span}\{Q_{v,V^-}:v\in R_I\}$.
Therefore,
\begin{equation}\label{eq:findJ-span}
\operatorname{span}\{Q_{v,V^-}:v\in R_I\}=L_I.
\end{equation}
Since $R_I$ is a union of connected components of $G_{V^+}$,
there are no edges between $R_I$ and $V^+\setminus R_I$.
Consequently, $N_G(R_I)\subseteq V^-$.
Moreover, every column of $Q_{R_I,V^-}$ indexed by
$V^-\setminus N_G(R_I)$ is zero. Hence, using~\eqref{eq:findJ-span}, we get:
\begin{equation}\label{eq:findJ-rank-support}
{\rm rank}\bigl(Q_{R_I,N_G(R_I)}\bigr)
=
{\rm rank}(Q_{R_I,V^-})
=
\dim(L_I) \le r.
\end{equation}
The algorithm considers every set $I\subseteq[s]$ with
$|I|\le r$ and $\dim(L_I)\le r$. For each such $I$, it computes the connected components of $\widehat G_I$. The candidate is discarded if
$|D\cap N_G(R_I)|>h$
for some connected component $D$ of $\widehat G_I$. For every
remaining candidate, it applies the treewidth algorithm
of~\cite{Korhonen21} to $\widehat G_I$ with parameter $k$. That
algorithm runs in time $2^{O(k)}n$ and either returns a tree
decomposition of width at most $2k+1$ or determines that
$\tw(\widehat G_I)>k$. If a tree decomposition is returned, the
algorithm outputs $J_I$. By construction every subset returned in this way satisfies the requirements of condition~1 of Theorem~\ref{th:oneSwitch}.

It remains to show that, whenever condition~1 of Theorem~\ref{th:oneSwitch} holds, the algorithm returns a subset $J\subseteq[s]$. Recall the subset $J^*$ whose existence follows from condition~1, and set
$$
L^*:=\operatorname{span}\{Q_{v,V^-}:v\in R^*\}.
$$
Since the columns of $Q_{R^*,V^-}$ outside $N_G(R^*)$ are zero, the rank assumption on $J^*$ gives
$\dim(L^*)
={\rm rank}\bigl(Q_{R^*,N_G(R^*)}\bigr)
\le r$.
Choose a basis of $L^*$ consisting of rows of $Q_{R^*,V^-}$,
and let $I^*$ be the set of indices of the components containing
these basis rows; take $I^*=\emptyset$ if $L^*=\{0\}$. Then
$|I^*|\le r$. 
Since $R^*$ is a union of components $C_j$, every
$C_i$ with $i\in I^*$ is contained in $R^*$. Hence all rows
associated with these components belong to $L^*$, whereas the
chosen basis rows span $L^*$. Consequently,
\begin{equation}\label{eq:findJ-Lstar}
L_{I^*}=L^*.
\end{equation}
We next claim that
\begin{equation}\label{eq:findJ-Jsubset}
J_{I^*}\subseteq J^*.
\end{equation}
Indeed, if $j\notin J^*$, then $C_j\subseteq R^*$, and therefore
$L_j\subseteq L^*=L_{I^*}$ by~\eqref{eq:findJ-Lstar}. Hence
$j\notin J_{I^*}$, proving~\eqref{eq:findJ-Jsubset}.
By~\eqref{eq:findJ-span} and \eqref{eq:findJ-Lstar},
$$
\operatorname{span}\{Q_{v,V^-}:v\in R_{I^*}\}
=
\operatorname{span}\{Q_{v,V^-}:v\in R^*\}.
$$
Fix $u\in V^-$. By the definition of the interaction graph,
$u\in N_G(R_{I^*})$ if and only if $Q_{vu}\neq 0$ for some
$v\in R_{I^*}$. 
By the equality of the two spans, this holds if and only if the
$u$-th coordinate is nonzero in some vector of
$\operatorname{span}\{Q_{v,V^-}:v\in R^*\}$, which is equivalent
to $u\in N_G(R^*)$. Therefore,
\begin{equation}\label{eq:findJ-neighborhood}
N_G(R_{I^*})=N_G(R^*).
\end{equation}
By~\eqref{eq:findJ-Jsubset}, every clique added to
construct $\widehat G_{I^*}$ is also added to construct $\widehat G^*$.
Thus $\widehat G_{I^*}$ is a subgraph of $\widehat G^*$, and
hence
$$
\tw(\widehat G_{I^*})\le\tw(\widehat G^*)\le k.
$$
Furthermore, every connected component $D_I$ of $\widehat G_{I^*}$
is contained in some connected component $D^*$ of $\widehat G^*$. By~\eqref{eq:findJ-neighborhood},
$$
|D_I\cap N_G(R_{I^*})|
\le
|D^*\cap N_G(R^*)|
\le h.
$$
Therefore the candidate $I^*$ satisfies
$|D_I\cap N_G(R_{I^*})|\le h$
for every connected component $D_I$ of $\widehat G_{I^*}$.
Moreover, $\tw(\widehat G_{I^*})\le k$. Hence the treewidth
algorithm returns a decomposition for $\widehat G_{I^*}$, and the algorithm outputs $J_{I^*}$.
Finally, the number of sets $I$ considered is at most
$$
\sum_{\ell=0}^{\min\{r,s\}}\binom{s}{\ell}
\le
\sum_{\ell=0}^r s^\ell
=
s^{O(r)}.
$$
Since $r$ is fixed, this is polynomial in $n$. For each candidate, all required linear-algebra and graph operations take polynomial time in $\L$, while the treewidth algorithm takes $2^{O(k)}n$ time. Since $k=O(\log n)$, the total running time is polynomial in the input length.
\end{proof}

Proposition~\ref{prop:findJ} shows that the subset $J$ in
Theorem~\ref{th:oneSwitch} need not be given explicitly, since a suitable choice can be found in polynomial time whenever condition~1 holds. For certain graph classes, however, a suitable choice of $J$ can be identified directly from the graph structure. The following corollary gives such a class.
Recall that a graph is a \emph{cactus graph} (resp. \emph{block graph}) if each of its maximal biconnected components (blocks) is an edge or a cycle (resp. clique). Moreover,
a graph is a \emph{block-cactus graph} if each of its maximal biconnected components is either a clique or a cycle.

\begin{corollary}\label{cor:blockCactus}
Let $G=(V,E)$ be the interaction graph of Problem~\eqref{eq:QP}.
Define
$V^+:=\{i\in V:q_{ii}>0\}$.
Suppose that $G$ is a block-cactus graph with the clique number $\omega(G)\in O(\log |V|)$. Moreover, suppose that $G_{V^+}$ is a forest. Let $R$ be the union of the connected components of $G_{V^+}$ containing at least one cut vertex of $G$.
Suppose that the following conditions hold:
\medskip
\begin{enumerate}
\item For every connected component $D$ of $G_{V\setminus R}$, we have $|D\cap N_G(R)|\in O(\log |V|)$.

\item ${\rm rank}(Q_{R,N_G(R)})\in O(1)$.
\end{enumerate}
\medskip
Then Problem~\eqref{eq:QP} can be solved in time polynomial in the input length.
\end{corollary}

\begin{proof}
Let $C_1,\dots,C_s$ denote the connected components of $G_{V^+}$,
and define
$$
J:=\{j\in[s]:C_j\text{ contains no cut vertex of }G\}.
$$
By the definition of $R$, we have $R=V^+\setminus\bigcup_{j\in J}C_j$.
We apply Theorem~\ref{th:oneSwitch} with this choice of $J$.
Fix $j\in J$. Since no vertex of $C_j$ is a cut vertex of $G$, all
vertices of $C_j$ belong to the same block $B$ of $G$. Moreover, every neighbor of $C_j$ also belongs to $B$. Indeed, a non-cut vertex of a graph belongs to a unique block, and the connectedness of $C_j$ therefore forces all vertices of $C_j$ and all edges incident to them to belong to the same block.
Since $G$ is a block-cactus graph, $B$ is either a clique or a cycle. If $B$ is a clique, then $N_G(C_j)$ is already a clique. If $B$ is a cycle, then $C_j$ is a path, since $G_{V^+}$ is a forest, and therefore
$|N_G(C_j)|\leq2$. The latter follows since every non-cut vertex of a cycle block has degree two.
Therefore, making $N_G(C_j)$ a clique either adds no edge
or adds the edge joining the two ends of the path through $C_j$.

We first verify that assumption~1 of Theorem~\ref{th:oneSwitch} holds.
Let $\widehat G$ be the graph on $V^-$ obtained from $G_{V^-}$ by
making $N_G(C_j)$ a clique for every $j\in J$. In a clique block, this operation adds no edge. In a cycle block, the added edge, if any, can be obtained by contracting the path through $C_j$. Consequently, $\widehat G$ is a minor of $G_{V\setminus R}$ and hence also a minor of $G$.
The treewidth of a block-cactus graph is the maximum of the
treewidths of its blocks. A clique block has treewidth one less than its cardinality, while a cycle block has treewidth at most two. Therefore,
$\tw(G)\leq\max\{\omega(G)-1,2\}\in O(\log |V|)$,
and hence $\tw(\widehat G)\in O(\log |V|)$.
Since $R$ is a union of
connected components of $G_{V^+}$,
$N_G(R)\subseteq V^-$.
For every connected component $D$ of $G_{V\setminus R}$, replacing
each $C_j$, $j\in J$, by a clique on $N_G(C_j)$ preserves
connectivity among the vertices of $D\cap V^-$. Consequently, the
connected components of $\widehat G$ are precisely the nonempty sets $D\cap V^-$, where $D$ ranges over the connected components of $G_{V\setminus R}$. Since $N_G(R)\subseteq V^-$,
we have $(D\cap V^-)\cap N_G(R)=D\cap N_G(R)$.
It follows from assumption~1 that every connected component
$\widehat D$ of $\widehat G$ satisfies
$|\widehat D\cap N_G(R)|\in O(\log |V|)$.
Thus assumption~1 of Theorem~\ref{th:oneSwitch} is satisfied.
Finally, since $G_{V^+}$ is a forest, every connected component $C_j$ of $G_{V^+}$ is a tree. Hence, by
Theorem~\ref{thm:main}, for every $j\in[s]$ and every
$d\in\Q^{C_j}$, the optimization problem~\eqref{eq:cor-simple-Cj}
can be solved in polynomial time. Thus assumption~2 of
Theorem~\ref{th:oneSwitch} is satisfied. Finally, assumption~2 of the corollary is precisely condition~1(ii) of
Theorem~\ref{th:oneSwitch}, and this completes the proof.
\end{proof}

\begin{remark}\label{rem:cactus-strict}
Corollary~\ref{cor:blockCactus} is not implied by either
Theorem~\ref{th:oneBlock} or
Theorem~\ref{th:smallNeighborhood}. To see this, consider the
following family of instances. For each $m\geq3$, let $
V^+:=\{r,b_1,\dots,b_m\}$ and $V^-:=\{w_1,\dots,w_m\}\cup\{z_1,\dots,z_m\}$.
Let $w_1,\dots,w_m$ induce a cycle, let $b_j$ be adjacent only to $w_j$ for every $j\in[m]$, and let $r$ be adjacent to $w_1$ and to every $z_i$, $i\in[m]$. There are no other edges.
It can be checked that the graph $G$ is a cactus, and $V^+$ is a stable set. The only cut vertex of $G$ belonging to $V^+$ is $r$. Hence, with the notation of Corollary~\ref{cor:blockCactus}, we have $R=\{r\}$.
The connected components of $G_{V\setminus R}$ consist of the
component $\{w_1,\dots,w_m\} \cup \{b_1,\dots,b_m\}$, and the isolated vertices $z_1,\dots,z_m$. Each of these components 
intersects $N_G(r)$ in exactly one vertex.
Therefore, assumption~1 of Corollary~\ref{cor:blockCactus} is satisfied. Moreover,
${\rm rank}(Q_{R,N_G(R)})\leq1$, so Corollary~\ref{cor:blockCactus} applies.
However, Theorem~\ref{th:oneBlock} does not apply. Indeed, the cycle $D:=\{w_1,\dots,w_m\}$ is a connected component of $G_{V^-}$ and, since $w_j$ is adjacent to $b_j\in V^+$ for every $j\in[m]$, we have
$D\cap N_G(V^+)=D$.
Thus $|D\cap N_G(V^+)|=m=\Theta(|V|)$,
violating assumption~1 of Theorem~\ref{th:oneBlock}. 
Finally, Theorem~\ref{th:smallNeighborhood} does not apply because $N_G(r)=\{w_1,z_1,\dots,z_m\}$.
In the graph obtained by making the neighborhood of every continuous component a clique, $N_G(r)$ therefore induces a clique of cardinality $m+1$. Hence the resulting graph has treewidth at least $m$, and assumption~1 of Theorem~\ref{th:smallNeighborhood} is violated.
Thus, Corollary~\ref{cor:blockCactus} contains families of instances that are covered by neither Theorem~\ref{th:oneBlock} nor Theorem~\ref{th:smallNeighborhood}.
\end{remark}

\ifthenelse {\boolean{MPA}}
{

\bigskip
\section*{Declarations}

\noindent
\textbf{Competing interests:}
A. Del Pia is on the Editorial Board of Mathematical Programming journal. The authors have no other competing interests to declare that are relevant to the content of this article.

} {
}

\medskip
\noindent
\textbf{Funding:}
A. Del Pia is partially funded by AFOSR grant FA9550-23-1-0433 and ONR grant N00014-25-1-2490. 
A. Khajavirad is in part supported by AFOSR grant FA9550-23-1-0123 and by ONR grant N00014-25-1-2491.
Any opinions, findings, and conclusions or recommendations expressed in this material are those of the authors and do not necessarily reflect the views of the Air Force Office of Scientific Research or the Office of Naval Research.

\ifthenelse {\boolean{MPA}}
{
\bibliographystyle{spmpsci}
\bibliography{biblio}
}
{
\begin{footnotesize}
\bibliographystyle{plain}
\bibliography{biblio}
\end{footnotesize}
}

\end{document}